\documentclass[12pt]{amsart}
\usepackage{amssymb,latexsym}
\usepackage{amscd}
\usepackage{verbatim}
\usepackage{amsmath,amsfonts,amssymb,epsf}
\usepackage{graphicx}

\def\ga{\alpha}     \def\gb{\beta}       \def\gg{\gamma}
       \def\gd{\delta}      
\def\gth{\theta}                         \def\vge{\varepsilon}
       \def\vgf{\varphi}    
            \def\gl{\lambda}
\def\gm{\mu}                 
    \def\gr{\rho}        
       \def\gt{\tau}

     \def\Gd{\Delta}

\def\Gw{\Omega}              

\newcommand{\pd}{\partial}

\newcommand{\Real}{\mathbb{R}}

\newcommand{\R}{\mathbb{R}}
\newtheorem{theorem}{Theorem}[section]
\newtheorem{lemma}[theorem]{Lemma}

\newtheorem{remark}[theorem]{Remark}

\newtheorem{corollary}[theorem]{Corollary}

\newtheorem{example}[theorem]{Example}
\newtheorem{definition}[theorem]{Definition}
\newtheorem{question}[theorem]{Question}
\newtheorem{conjecture}[theorem]{Conjecture}
\newtheorem{thm}[theorem]{Theorem}
\newtheorem{rem}[theorem]{Remark}
\newtheorem{lem}[theorem]{Lemma}

\newtheorem{Cor}[theorem]{Corollary}

\newcommand{\Green}[4]{\mbox{$G^{#1}_{#2}(#3,#4)$}}
\newcommand{\dx}{\,\mathrm{d}x}

\newcommand{\dy}{\,\mathrm{d}y}
\newcommand{\dz}{\,\mathrm{d}z}
\newcommand{\dt}{\,\mathrm{d}t}

\newcommand{\be}{\begin{equation}}
\newcommand{\ee}{\end{equation}}
\newcommand{\mysection}[1]{\section{#1}\setcounter{equation}{0}}
\newcommand{\bea}{\begin{eqnarray}}
\newcommand{\eea}{\end{eqnarray}}
\newcommand{\bean}{\begin{eqnarray*}}
\newcommand{\eean}{\end{eqnarray*}}

\def\squarebox#1{\hbox to #1{\hfill\vbox to #1{\vfill}}}

\begin{document}
\renewcommand{\theequation}{\thesection.\arabic{equation}}
\title[Large time behavior of the heat kernel]%
{Some aspects of large time behavior of the heat kernel: an overview with perspectives}
\author{Yehuda Pinchover}
\address{Department of Mathematics\\ Technion - Israel Institute of
Technology\\ Haifa 32000, Israel}
\email{pincho@techunix.technion.ac.il}
 %
\date{August 31, 2012}
\dedicatory{Dedicated to Professor Michael Demuth on the occasion of his 65th birthday}
\begin{abstract}
We discuss a variety of developments in the study of large time behavior of the positive minimal heat kernel of a time independent (not necessarily symmetric) second-order parabolic operator defined on a domain $M\subset\R^d$, or more generally, on a noncompact Riemannian manifold $M$. Our attention is mainly focused on {\em general} results in general settings.
\end{abstract}
\maketitle
\tableofcontents
\mysection{Introduction}\label{Introduction}
The large time behavior of the heat kernel of a second-order parabolic operator has been extensively studied over the recent decades (see for example the following monographs and
survey articles
\cite{A,Ch,C,C98,Dheat,G,Has,Li,Pinsky,PE,R,Simon82,Va,VSCC,W}, and
references therein). The purpose of the present
paper is to review a variety of developments in this area, and to point out a number of their consequences. Our attention is mainly focused on {\em general} results in general settings. Still, the selection of topics in this survey is incomplete, and is according to the author's working experience and taste. The reference list is far from being complete and serves only this expos\'{e}.

Let $P$ be a general linear, second-order, elliptic operator defined on a domain $M\subset\R^d$ or, more generally, on a noncompact, connected, Riemannian manifold $M$ of dimension $d\geq 1$. Denote the cone of all positive solutions of the
equation $Pu=0$ in $M$ by $\mathcal{C}_{P}(M)$. The {\em
generalized principal eigenvalue}  is defined by
$$\gl_0=\gl_0(P,M)
:= \sup\{\gl \in \mathbb{R} \; \mid\; \mathcal{C}_{P-\lambda}(M)\neq
\emptyset\}.$$ Throughout this paper we always assume that
$\lambda_0>-\infty$.

Suppose that $\lambda_0\geq 0$, and consider the (time-independent) parabolic operator
\begin{equation}\label{eqL}
  Lu:=\pd_t u+P(x,\pd_x)u \qquad  (x,t)\in M\times (0,\infty).
\end{equation}
We denote by $\mathcal{H}_P(M\times
(a,b))$ the cone of all nonnegative solutions of the parabolic equation
\begin{equation}\label{Lu0}
Lu =0 \qquad \mbox{ in } M\times(a,b).
\end{equation}

Let $k_P^{M}(x,y,t)$ be the {\em positive minimal heat kernel} of the
parabolic operator $L$ on the manifold $M$.  By definition, for a fixed $y\in M$, the function $(x,t) \mapsto k_P^M(x,y,t)$
is the minimal positive solution of the equation
\begin{equation}\label{heat}
 Lu =0  \qquad\mbox{ in } M\times(0,\infty) ,
\end{equation}
subject to  the initial data $\delta_y$, the Dirac distribution at $y \in M$.
It can be easily checked that for $\lambda\leq \lambda_0$, the
heat kernel $k_{P-\lambda}^M$ of the operator $P-\lambda$ on $M$ satisfies the identity
\begin{equation}\label{eq_hkp-gl}
k_{P-\lambda}^M(x,y,t)=e^{\lambda t}k_P^M(x,y,t).
\end{equation}
So, it is enough to study the large time behavior of $k_{P-\lambda_0}^M$, and therefore, in most cases we assume that $\gl_0(P,M)=0$. Note that the heat kernel of the operator $P^*$, the formal adjoint of the operator $P$ on $M$, satisfies the relation
$$k_{P^*}^M(x,y,t)=k_P^M(y,x,t).$$

\vskip 3mm

Here we should mention that many authors derived upper and lower Gaussian bounds for
heat kernels of elliptic operators on $M:=\mathbb{R}^d$, or more
generally, on noncompact Riemannian manifolds $M$. As a prototype result, let us recall the following classical result of Aronson \cite{A}:
 \begin{example}\label{ex_Aronson} {\em Let $P$ be a second-order uniformly elliptic operator in divergence form
on $\mathbb{R}^d$ with real coefficients satisfying some general boundedness assumptions. Then the
following Gaussian estimates hold:
\begin{multline*}
C_1(4\pi t)^{-d/2}
\exp\left(-C_2\frac{|x-y|^2}{t}-\omega_1 t\right)\leq k_P^{\R^d}(x,y,t)\\[2mm]
\leq C_3(4\pi t)^{-d/2}\exp\left(-C_4\frac{|x-y|^2}{t}+\omega_2 t\right)
 \quad \forall (x,y,t)\in \mathbb{R}^d
 \times\mathbb{R}^d\times\mathbb{R}_+.
\end{multline*}
 }
\end{example}
However, since Gaussian  estimates of the above type in general are not tight as $t\to\infty$, such bounds  do not provide us with the exact large time behavior of the heat kernel, let alone strong ratio limits of two heat kernels.

\vskip 3mm

In spite of this, and as a first and rough result concerning the large time behavior of the heat kernel,  we have the following explicit and useful formula
  \begin{equation}\label{logform}
\lim_{t \to \infty} \dfrac{\log k_P^M(x,y,t)}{t} =-\lambda_0.
\end{equation}
We note that \eqref{logform} holds in the general case, and characterizes the generalized principal eigenvalue $\gl_0$ in terms of the large time behavior of $\log k_P^M(x,y,t)$. The above formula is well known in the symmetric case, see for example \cite{PLi86} and \cite[Theorem~10.24]{Grigoryan}. For the proof in the general case, see Corollary~\ref{asser_log}.

\vskip 3mm

To get a more precise result one should introduce the notion of criticality. We say that the operator~$P$ is \emph{subcritical} (respectively, \emph{critical}) in $M$ if for some $x \not = y$, and therefore for any
$x \not = y$,  $x,y\in M$, we have
\begin{equation}\label{def.critical}
  \int_0^\infty k_P^{M}(x,y,\tau)\,\mathrm{d}\tau<\infty \qquad
  \left(\mbox{respectively, } \int_0^\infty
  k_P^{M}(x,y,\tau)\,\mathrm{d}\tau=\infty\right).
\end{equation}
It follows from the above definition that, roughly speaking, the heat kernel of a subcritical operator in $M$ ``decays" faster as $t\to\infty$ than the heat kernel of a critical operator in $M$. This rule of thumb will be discussed in Section~\ref{sect1DMY}.

\vskip 3mm

%
If $P$ is subcritical in $M$, then the function
\begin{equation}\label{def.Gr}
G_P^{M}(x,y):=  \int_0^\infty k_P^{M}(x,y,\tau)\,\mathrm{d}\tau
\end{equation}
is called the {\em positive minimal Green function} of the operator $P$ in $M$.

It follows from \eqref{logform} that for $\lambda<\lambda_0$, the operator $P-\lambda$ is
subcritical in $M$. Clearly, $P$ is subcritical (respectively,
critical) in $M$, if and only if $P^*$, is subcritical
(respectively, critical) in $M$. Furthermore, it is well known that if $P$ is
critical in $M$, then $\mathcal{C}_{P}(M)$ is a one-dimensional
cone, and any positive supersolution of the  equation $Pu=0$ in $M$ is in fact a solution. In this case, the unique positive solution $\varphi \in \mathcal{C}_{P}(M)$ is called
{\em Agmon ground state} (or in short ground state) of the operator $P$ in $M$ \cite{Agmon82,Pheat,Pinsky}.
We denote the ground state of $P^*$ by $\varphi^*$.

\vskip 3mm

The following example demonstrates two prototype behaviors as $t\to \infty$ of heat kernels corresponding to two particular classical cases.
\begin{example}\label{ex_int}
{\em

1. Let $P=-\Delta$, $ M=\mathbb{R}^d$, where $d\geq 1$. It is well known that $\lambda_0(-\Gd,\R^d)=0$, and that the heat kernel is given by the Gaussian kernel. Hence,
\begin{multline*}
\mathrm{e}^{\lambda_0 t} k_{-\Delta}^{\mathbb{R}^d}(x,y,t)=k_{-\Delta}^{\mathbb{R}^d}(x,y,t)\\=\frac{1}{(4\pi t)^{d/2}}\exp \left(\frac{-|x-y|^2}{4t}\right)  \underset{t\to \infty}{\sim} t^{-d/2} \underset{t \to \infty }{\to} 0.
\end{multline*}
So, the rate of the decay depends on the dimension, and clearly, $-\Gd$ is critical in $\mathbb{R}^d$ if and only if $d=1,2$. Nevertheless, the above limit is $0$ in any dimension.

\vskip 3mm

2. Suppose that $P$ is a symmetric nonnegative elliptic operator with real and smooth coefficients which is defined on a smooth bounded domain $M\Subset\R^d$, and let  $\{\varphi_n\}_{n=0}^\infty$ be the complete orthonormal sequence of the (Dirichlet) eigenfunctions  of $P$ with the corresponding nondecreasing sequence of eigenvalues $\{\gl_n\}_{n=0}^\infty$. Then the heat kernel has the eigenfunction expansion
    \begin{equation}\label{ex_ef_exp}
    k_P^M(x,y,t) = \sum_{n=0}^\infty e^{-\lambda_n t}\varphi_n(x)\varphi_n(y).
\end{equation}
 Hence,
$$\mathrm{e}^{\lambda_0 t} k_{P}^{M}(x,y,t)=\sum_{n=0}^\infty \mathrm{e}^{(\lambda_0-\lambda_n) t}\varphi_n(x)\varphi_n(y)
\underset{t \to \infty }{\to} \varphi_0(x)\varphi_0(y)>0 .$$
Therefore, the operator $P-\gl_0$ is critical in $M$, $t^{-1}\log k_{P-\gl_0}^M(x,y,t)\underset{t \to \infty }{\to} 0$, but $k_{P-\gl_0}^M$ does not decay as $t\to \infty$.

\vskip 3mm

3. Several explicit formulas of heat kernels of certain classical operators are included in \cite{CCFI,Dheat,Grigoryan}, each of which
either tends to zero or converges to a positive function as $t\to\infty$.}
\end{example}

The characterization of $\gl_0$ in terms of the large time behavior of the heat kernel, given by \eqref{logform}, provides us with the asymptotic behavior of $\log k_P^{M}$ as $t\to \infty$ but not of $k_P^{M}$ itself. Moreover, \eqref{logform} does not distinguish between critical and subcritical operators. In the first part of the present article we provide a complete proof that
$$\lim_{t\to\infty} e^{\lambda_0 t}k_P^{M}(x,y,t)$$ always
exists. This basic result has been proved in two parts in \cite{Pheat} and \cite{P03},
and here, for the first time, we give a comprehensive and a bit simplified proof (see also \cite{AB,CK,GS,Has,KLVW,Pinsky,S,Va,Z} and references therein for previous and related results). We have:
\begin{theorem}[\cite{Pheat,P03}]\label{thm1} Let $P$ be an elliptic operator defined on $M$, and assume that $\gl_0(P,M)\geq 0$.
\begin{description}
\item[{\em (i)}] {\bf The subcritical case:} If $P-\lambda_0$ is subcritical in $M$, then
$$\lim_{t\to\infty} e^{\lambda_0 t} k_P^{M}(x,y,t)=0.$$
  \item[{\em (ii)}] {\bf The positive-critical case:}  If $P-\lambda_0$ critical is in $M$, and the
ground states $\varphi$ and $\varphi^*$ of $P-\lambda_0$ and
$P^*-\lambda_0$, respectively, satisfy $\varphi^*\varphi \in
L^1(M)$, then
  $$\lim_{t\to\infty} e^{\lambda_0 t} k_P^{M}(x,y,t)=
  \frac{\varphi(x)\varphi^*(y)}{\int_M \varphi^*(z)\varphi(z)\dz}\,.$$
\item[{\em (iii)}] {\bf The null-critical case:} If $P-\lambda_0$ is critical in $M$, and the
ground states $\varphi$ and $\varphi^*$ of $P-\lambda_0$ and
$P^*-\lambda_0$, respectively, satisfy $\varphi^*\varphi\not \in
L^1(M)$, then
$$\lim_{t\to\infty}
  e^{\lambda_0 t} k_P^{M}(x,y,t)=0.$$
\end{description}

Moreover, for $\lambda<\lambda_0$, let $\Green{M}{P-\lambda}{x}{y}$ be the minimal positive
Green function of the elliptic operator $P-\lambda$ on $M$. Then  the following Abelian-Tauberian relation holds
\begin{equation}\label{eqgreen}
\lim_{t\to\infty} e^{\lambda_0 t}k_P^{M}(x,y,t)=
\lim_{\lambda\nearrow\lambda_0}(\lambda_0-\lambda)\Green{M}{P-\lambda}{x}{y}.
\end{equation}
\end{theorem}

\vspace{3mm}

The outline of the present paper is as follows. In Section~\ref{sec1} we provide a short review of the theory of positive solutions and the basic properties of the heat kernel. The proof of Theorem~\ref{thm1} is postponed to Section~\ref{secmainthm} since it needs some preparation. It turns out that the proof of the null-critical case (given in \cite{P03}) is the most subtle part of the proof of Theorem~\ref{thm1}. It relies on the large time behaviors of the parabolic capacitory potential and
of the heat content that are studied in Section \ref{secaux}, and on an extension of Varadhan's lemma that is proved in Section \ref{secVa} (see Lemma \ref{lemVa}). Section~\ref{sec_appl} is devoted to some applications to Theorem~\ref{thm1}. In particular, Corollary~\ref{cor_l3} seems to be new.

The next two sections deal with ratio limits. In Section~\ref{sect1DC}, we discuss the existence of the strong ratio limit
$$\lim_{t\to\infty}\frac{k_P^ M(x,y,t)}{k_P^ M(x_0,x_0,t)}$$
(Davies' conjecture), and in Section~\ref{sect1DMY} we deal with the conjecture that
\begin{equation}\label{eqconjMain_int}
\lim_{t\to\infty}\frac{k_{P_+}^M(x,y,t)}{k_{P_0}^M(x,y,t)}=0 \qquad \forall x,y\in M,
\end{equation}
where $P_+$ and $P_0$ are respectively subcritical and critical
operators in $M$. Finally, in Section~\ref{secequiv} we discuss the equivalence of heat kernels of two subcritical elliptic operators (with the same $\gl_0$) that agree outside a compact set. Unlike the analogous question concerning the equivalence of Green functions, the problem concerning the equivalence of heat kernels is still ``terra incognita". Solving this problem will in particular enable us to study the stability of the large time behavior of the heat kernel under perturbations.

The survey is an expanded version of a talk given by the author at the conference ``Mathematical Physics, Spectral Theory and Stochastic Analysis" held in Goslar, Germany, September 11--16, 2011, in honor of Professor Michael Demuth. We note that most of the results in sections~\ref{secaux}--\ref{sec_appl} originally appeared in \cite{Pheat,P03}, those of Section~\ref{sect1DC} appeared in \cite{PDavies}, while those in sections~\ref{sect1DMY}--\ref{secequiv} originally
appeared in \cite{fkp}.

\mysection{Preliminaries}\label{sec1} In this section we recall
basic definitions and facts concerning the theory of nonnegative
solutions of second-order linear elliptic and parabolic operators
(for more details and proofs, see for example \cite{Pinsky}).

 Let $P$ be a linear, second-order, elliptic operator defined in a  noncompact, connected,
$C^3$-smooth Riemannian manifold $M$ of dimension $d$. Here $P$ is
an  elliptic operator with real and H\"{o}lder continuous
coefficients which in any coordinate system
$(U;x_{1},\ldots,x_{d})$ has the form \be \label{P}
P(x,\partial_{x})=-\sum_{i,j=1}^{d}
a_{ij}(x)\partial_{i}\partial_{j} + \sum_{i=1}^{d}
b_{i}(x)\partial_{i}+c(x),
\end{equation}
where $\partial_{i}=\partial/\partial x_{i}$. We assume that for
every $x\in M$ the real quadratic form
\be
\label{ellip} \sum_{i,j=1}^{d} a_{ij}(x)\xi_{i}\xi_{j},\;\;
\xi=(\xi_{1},\ldots,\xi_{d}) \in \Real ^d
\end{equation}
is symmetric and positive definite.

Throughout the paper we always assume that
$\lambda_0(P,M)\geq 0$. In this case we simply say that the operator $P$ is {\em nonnegative in} $M$, and write $P \geq  0$ in $M$. So, $P \geq  0$ in $M$ if and only if the equation $Pu=0$ admits a global positive solution in $M$. We note that in the symmetric case, by the Agmon-Allegretto-Piepenbrink theory \cite{Agmon82}, the nonnegativity of  $P$ in $M$  is equivalent to the nonnegativity of the associated quadratic form on $C_0^\infty(M)$.

We consider the parabolic operator $L$
\begin{equation}\label{eqL1}
  Lu=u_t+Pu \qquad \mbox{ on } M\times (0,\infty),
\end{equation}
and the corresponding homogeneous equation $Lu=0$ in $M\times (0,\infty)$.
\begin{remark}
{\em We confine ourselves to classical solutions. However, since the
results and the proofs throughout the paper rely on standard elliptic and parabolic regularity, and basic
properties and results of potential theory, the results of the
paper are also valid for weak solutions in the case where the
elliptic operator $P$ is in divergence form
\begin{equation} \label{div_P}
Pu=-\sum_{i=1}^{d}  \partial_i\left[\sum_{j=1}^{d} a_{ij}(x)\partial_j u +  u\tilde{b}_i(x) \right]  +
 \sum_{i=1}^{d} b_i(x)\partial_i u   +c(x)u,
\end{equation}
with coefficients
which satisfy standard local regularity assumptions (as for
example in Section 1.1 of \cite{Mskew}). The results are also
valid in the framework of strong solutions, where the strictly
elliptic operator $P$ is of the form \eqref{P} and has locally bounded coefficients; the proofs
differ only in minor details from the proofs given here.}
\end{remark}
We write $\Omega_1 \Subset \Omega_2$ if $\Omega_2$ is open, $\overline{\Omega_1}$ is
compact and $\overline{\Omega_1} \subset \Omega_2$.
Let $f,g \in C(\Omega)$ be nonnegative functions; we use the notation $f\asymp g$ on
$\Omega$ if there exists a positive constant $C$ such that
$$C^{-1}g(x)\leq f(x) \leq Cg(x) \qquad \mbox{ for all } x\in \Omega.$$

Let $\{M_{j}\}_{j=1}^{\infty}$ be an {\em exhaustion} of  $M$,
i.e. a sequence of smooth, relatively compact domains in $M$ such that
$M_1\neq \emptyset$, $M_j\Subset M_{j+1}$ and
$\cup_{j=1}^{\infty}M_{j}=M$. For every $j\geq 1$, we denote
$M_{j}^*=M\setminus \mbox{cl}({M_j})$. Let
$M_\infty=M\cup\{\infty\}$ be the one-point compactification of
$M$. By a neighborhood of infinity in $M$ we mean a neighborhood of $\infty$ in $M_\infty$, that is, a set of the form $M\setminus K$, where $K$ is compact in $M$.

Assume that $P\geq 0$ in $M$. For every $j\geq 1$, consider the Dirichlet heat kernel
$k_P^{M_j}(x,y,t)$ of the parabolic operator $L=\partial_t+P$ in
$M_j$. So, for every continuous function $f$ with a compact
support in $M$, the function $$u(x,t):=\int_{M_j} k_P^{M_j}(x,y,t)f(y)\dy$$
solves the initial-Dirichlet boundary value problems
 \bea\label{eqibvpj}
 Lu&=&0 \quad \mbox{ in } M_j\times (0,\infty),\nonumber\\
  u&=&0 \quad \mbox{ on  } \partial M_j\times (0,\infty),\\
  u&=&f  \quad \mbox{ on  }  M_j\times \{0\}.\nonumber
  \eea
   By the generalized maximum principle, $\{k_P^{M_j}(x,y,t)\}_{j=1}^{\infty}$ is
an increasing sequence which converges to $k_P^{M}(x,y,t)$, the
{\em positive minimal heat kernel} of the parabolic operator $L$ in $M$.

The main properties of the positive minimal heat kernel are summarized in the following lemma.
\begin{lemma}\label{lem_hkp}
Assume that $P\geq 0$ in $M$. The heat kernel $k_P^{M}(x,y,t)$ satisfies the following properties.
\begin{enumerate}
\item Positivity: $k_P^M(x,y,t)\geq 0$ for all $t\geq 0$ and $x,y\in M$.

\vskip 3mm

\item For any $\gl\leq \gl_0$ we have
\begin{equation}\label{eq_shift}
k_{P-\lambda}^M(x,y,t)=e^{\lambda t}k_P^M(x,y,t).
\end{equation}

\vskip 3mm

\item The heat kernel satisfies the semigroup identity
\begin{equation}\label{eq_sgi}
\int_\Gw k_P^M(x,z,t)k_P^M(z,y,\gt) \dz= k_P^M(x, y, t+\gt) \;\; \forall t,\gt>0, \mbox{and } x,y\in M.
    \end{equation}

\item Monotonicity: If $M_1\subset M$, and $V_1\geq V_2$, then
$$k_P^{M_1}\leq k_P^{M},\qquad \mbox{and }\quad k_{P+V_1}^{M}\leq k_{P+V_2}^{M}.$$

\vskip 3mm

\item For any fixed $y\in M$ (respectively $x\in M$), $k_P^M(x,y,t)$ solves the equation $u_t+P(x,\pd_x)u=0$
(respectively $u_t+P^*(y,\pd_y)u=0$) in $M\times (0,\infty)$.

\vskip 3mm

\item Skew product operators: Let $M=M_1\times M_2$ be a product of two manifolds, and denote $x=(x_1,x_2)\in M=M_1\times M_2$. Consider a skew product elliptic operator of the form
$P:= P_1\otimes I_2+I_1\otimes P_2$, where $P_i$ is a second-order elliptic operator on $M_i$ satisfying the assumptions of the present paper, and $I_i$ is the identity map on $M_i$, $i=1,2$.
Then
\begin{equation}\label{eq_prod}
k_{P}^{M}(x,y,t)=
k_{P_1}^{M_1}(x_1,y_1,t)k_{P_2}^{M_2}(x_2,y_2,t).
\end{equation}

\vskip 3mm

\item Eigenfunction expansion: Suppose that $P$ is a symmetric elliptic operator with (up to the boundary) smooth coefficients which is defined on a smooth bounded domain $M$. Let $\{\varphi_n\}_{n=0}^\infty$ be the complete orthonormal sequence of the (Dirichlet) eigenfunctions  of $P$ with the corresponding nondecreasing sequence of eigenvalues $\{\gl_n\}_{n=0}^\infty$. Then the heat kernel of $P$ in $M$ has the eigenfunction expansion
    \begin{equation}\label{eq_ef_exp}
    k_P^M(x,y,t) = \sum_{n=0}^\infty e^{-\lambda_n t}\phi_n(x)\phi_n(y).
\end{equation}
\vskip 3mm

\item Let $v\in \mathcal{C}_{P}(M)$ and $v^*\in
\mathcal{C}_{P^*}(M)$. Then we have \cite{Pheat,PSt}
\begin{equation}\label{eqinvar1}
\int_M k_P^M(x,y,t)v(y)\dy\leq v(x),\quad  \mbox{and}\quad  \int_M
k_P^M(x,y,t)v^*(x)\dx\leq v^*(y).
\end{equation}
Moreover, (by the maximum principle) either $$\int_M k_P^M(x,y,t)v(y)\dy< v(x) \qquad \forall (x,t)\in M\times(0,\infty),$$ or  $$\int_M k_P^M(x,y,t)v(y)\dy= v(x) \qquad \forall (x,t)\in M\times (0,\infty),$$ and in the latter case $v$ is called an {\em invariant solution} of the operator $P$ on $M$ (see for example
\cite{D1,G1,Pheat,Pinsky}).

    \end{enumerate}
\end{lemma}

\begin{remark}\label{rem_inv}{\em
1. If there exists $v\in \mathcal{C}_{P}(M)$ such that $v$ is not invariant, then the positive Cauchy problem
$$Lu=0, \; u\geq 0 \quad \mbox{on } M\times (0,\infty),\quad u(x,0)=0\quad x\in M$$
does not admit a unique solution.

\vskip 3mm

2. An invariant solution is sometimes called {\em complete}. If
the constant function is an invariant solution with respect to the heat operator, then one says that the heat operator {\em conserves probability}, and the corresponding diffusion process is said to be {\em stochastically complete} \cite{Grigoryan}.
 }
\end{remark}

\begin{definition}\label{d_min_gr}{\em
Let $w$ be a positive solution of the equation $Pu=0$ in $M\setminus K$, where $K\Subset M$. We say that $w$ is a positive solution of the
equation $Pu=0$  of {\em minimal growth in a neighborhood of
infinity in} $M$ if for any  $v\in C(\mbox{cl}(M_j^*))$ which is a
positive supersolution of the equation $Pu=0$ in $M_j^*$ for some
$j\geq 1$ large enough, and satisfies $w\leq v$ on $\partial M_j^*$, we have
$w\leq v$ in $M_j^*$ \cite{Agmon82,Pheat,Pinsky}.
 }
 \end{definition}

In the subcritical (respectively, critical) case, the Green function $G_P^{M}(\cdot,y)$ (respectively, the ground state $\vgf$) is a  positive solution of the equation $Pu=0$ of minimal growth in a neighborhood of
infinity in $M$. Recall that if $\lambda<\lambda_0$, then $P-\lambda$ is
subcritical in $M$. In particular, if $P$ is critical in $M$, then $\gl_0=0$.

\vskip 3mm

Next, we recall the parabolic Harnack inequality. We denote by $Q(x_0,t_0,R,\gt)$ the parabolic box
$$Q(x_0,t_0,R,\gth):=\{(x,t)\in M\times \R\mid \gr(x,x_0)<R, t\in (t_0,t_0+\gth R^2)\},$$
where $\gr$ is the given Riemannian metric on $M$. We have:
\begin{lemma}[Harnack inequality]\label{lem_Harnack}
Let $u$ be a nonnegative solution of the equation $Lu=0$ in $Q(x_0,t_0,R,\gth)$, and assume that $\gth>1$ and $0<R<R_0$. Then
\begin{equation}\label{eq_harnack}
    u(x_0,t_0+R^2)\leq C u(x,t_0+\gth R^2)
    \end{equation}
for every $\gr(x,x_0)<R/2$, where $C=C(L,d,R_0,\gth)$. Moreover, $C$ varies within bounded bounds for all $\gth>1$ such that $0< \vge\leq (\gth-1)^{-1}\leq M$.
\end{lemma}

Let $v\in \mathcal{C}_{P}(M)$ and $v^*\in
\mathcal{C}_{P^*}(M)$. Then the parabolic Harnack inequality \eqref{eq_harnack}, and (\ref{eqinvar1}) imply the following important estimate
\begin{equation}\label{eqHar1}
k_P^M(x,y,t)\leq c_1(y)v(x), \qquad {\and }\; k_P^M(x,y,t)\leq
c_2(x)v^*(y)
\end{equation}
for all $x,y\in M$ and $t>1$ (see \cite[(3.29--30)]{Pheat}). Recall that in
the critical case, by uniqueness, $v$ and $v^*$ are (up to a multiplicative constant) the ground states
$\varphi$ and $\varphi^*$ of $P$ and $P^*$ respectively. Moreover, it is known that the ground state $\varphi$ (respectively,
$\varphi^*$) is a positive invariant solution of the operator $P$
(respectively, $P^*$) in $M$. So,
\begin{equation}\label{eqinvar}
\int_M k_P^M(x,y,t)\varphi(y)\dy= \varphi(x),\;\;  \mbox{and }
\int_M k_P^M(x,y,t)\varphi^*(x)\dx= \varphi^*(y).
\end{equation}

\vskip 3mm
We distinguish between two types of criticality.
\begin{definition}\label{defnull}{\em
A critical operator $P$ is said to be {\em positive-critical} in
$M$ if $\varphi^*\varphi\in L^1(M)$, and {\em null-critical} in
$M$ if $\varphi^*\varphi\not\in L^1(M)$.
 }\end{definition}
\begin{remark}{\em Let $\mathbf{1}$ be the constant function on $M$,
taking at any point $x\in M$ the value $1$. Suppose that
$P\mathbf{1}=0$. Then $P$ is subcritical (respectively,
positive-critical, null-critical) in $M$ if and only if the
corresponding diffusion process is transient (respectively,
positive-recurrent, null-recurrent). For a thorough discussion of the probabilistic interpretation of criticality theory,  see \cite{Pinsky}.

 In fact, in the critical case it is natural to use the well known (Doob)
$h$-transform with $h=\varphi$, where $\varphi$ is the ground
state of $P$. So,
$$P^\varphi u:=\frac{1}{\varphi}P(\varphi u)\quad\mbox{ and therefore}\quad
k_{P^\varphi}^M(x,y,t)=\frac{1}{\varphi(x)}k_{P}^M(x,y,t)\varphi(y).$$

Clearly, $P^\varphi$ is an elliptic operator which satisfies all
our assumptions. Note that $P^\varphi$ is null-critical
(respectively, positive-critical) if and only if $P$ is
null-critical (respectively, positive-critical), and the ground
states of $P^\varphi$ and $(P^\varphi)^*$ are $\mathbf{1}$ and
$\varphi^*\varphi$, respectively. Moreover,
$$\lim_{t\to\infty}k_{P^\varphi}^M(x,y,t)=0 \quad \mbox{ if and
only if } \quad
 \lim_{t\to\infty}k_{P}^M(x,y,t)=0.$$ Therefore, in the critical case, we may assume that $$\mbox{{\bf
(A)}}\qquad \qquad\qquad P\mathbf{1}=0, \mbox{ and } P \mbox{ is a
critical operator in } M.\qquad\qquad\mbox{}$$
 }\end{remark}

It is well known that on a general noncompact manifold $M$, the
solution of the Cauchy problem for the parabolic equation $Lu=0$
is not uniquely determined (see for example \cite{IM} and the
references therein). On the other hand, under Assumption {\bf
(A)}, there is a unique {\em minimal} solution of the Cauchy
problem and of certain initial-boundary value problems for {\em
bounded} initial and boundary conditions. More precisely,
\begin{definition}\label{defCp}{\em Assume that $P\mathbf{1}=0$.
Let $f$ be a {\em bounded} continuous function on $M$.  By the {\em
minimal solution} $u$ of the Cauchy problem \bean\label{eqCp}
 Lu&=&0 \quad \mbox{ in } M\times (0,\infty),\\
  u&=&f  \quad \mbox{ on  }  M\times \{0\},
  \eean
we mean the function
\begin{equation}\label{eqminim}
u(x,t):=\int_{M} k_P^{M}(x,y,t)f(y)\dy.
 \end{equation}
 }\end{definition}
 Note that \eqref{eqHar1} implies that $u$ in \eqref{eqminim} is well defined.
\begin{definition}\label{defibvp}{\em Assume that $P\mathbf{1}=0$.
Let $B\Subset M_1$ be a smooth bounded domain such that
$B^*:=M\setminus \mbox{cl}(B)$ is connected. Assume that $f$ is a
bounded continuous function on $B^*$, and $g$ is a bounded
continuous function on $\partial B\times (0,\infty)$.  By the {\em
minimal solution} $u$ of the initial-boundary value problem
 \bea\label{eqibvpfg}
 Lu&=&0 \quad \mbox{ in } B^*\times (0,\infty),\nonumber\\
  u&=&g \quad \mbox{ on  } \partial B\times (0,\infty),\\
  u&=&f  \quad \mbox{ on  }  B^*\times \{0\}\nonumber,
  \eea
we mean the limit of the solutions $u_j$ of the following
initial-boundary value problems \bean\label{eqibvpfgk}
 Lu&=&0 \quad \mbox{ in } (B^*\cap M_j)\times (0,\infty),\\
  u&=&g \quad \mbox{ on  } \partial B\times (0,\infty),\\
    u&=&0 \quad \mbox{ on  } \partial M_j\times (0,\infty),\\
  u&=&f  \quad \mbox{ on  }  (B^*\cap M_j)\times \{0\}.
  \eean
}\end{definition}
\begin{remark}{\em
It can be easily checked that the sequence $\{u_j\}$ is indeed a
converging sequence which converges  to a solution of the
initial-boundary value problem (\ref{eqibvpfg}).
 }\end{remark}
Next, we recall some results concerning the theory of positive solutions of {\em elliptic} equations that we shall need in the sequel.

The first result is a {\em Liouville comparison theorem} in the symmetric case.
\begin{theorem}[\cite{Pliouv}]\label{mainthmLt}
Let $P_0$ and $P_1$ be two symmetric operators defined on $M$ of the form
\begin{equation}\label{PS}
P_ju=-m_j^{-1} \mathrm{div}(m_jA_j\nabla u) +V_ju \qquad j=0,1.
\end{equation}
Assume that the following assumptions hold true.
\begin{itemize}
\item[(i)] The operator  $P_0$ is critical in $M$. Denote
by $\varphi\in \mathcal{C}_{P_0}(M)$ its ground state.

\item[(ii)]  $P_1\geq 0$ in $M$, and there exists a
real function $\psi\in H^1_{\mathrm{loc}}(M)$ such that
$\psi_+\neq 0$, and $P_1\psi \leq 0$ in $M$, where
$u_+(x):=\max\{0, u(x)\}$.

\item[(iii)] The following matrix inequality holds
\begin{equation}\label{psialephia}
(\psi_+)^2(x) m_1(x)A_1(x)\leq C\varphi^2(x) m_0(x)A_0(x)\qquad  \mbox{for a.e. }
x\in  M,
\end{equation}
where $C>0$ is a positive constant.
\end{itemize}
Then the operator $P_1$ is critical in $M$, and $\psi$ is its
ground state. In particular, $\dim \mathcal{C}_{P_1}(M)=1$
and $\lambda_0(P_1,M)=0$.
\end{theorem}
In the sequel we shall also need to use results concerning small and semismall perturbations of a subcritical elliptic operator. These notions were introduced in \cite{P89} and \cite{Msemismall}
respectively, and are closely related to the stability of
$\mathcal{C}_P(\Omega)$ under perturbation by a potential $V$.
\begin{definition} \label{spertdef}{\em
Let $P$ be a subcritical operator in $M$, and let $V$ be a
real valued potential defined on $M$.

\vskip 2mm

{\em (i)} We say that $V$ is a {\em small perturbation}
 of $P$ in $M$ if
\be \label{sperteq} \lim_{j\rightarrow \infty}\left\{\sup_{x,y\in
M_{j}^*} \int_{M_{j}^*}\frac{\Green{M}{P}{x}{z}|V(z)|
\Green{M}{P}{z}{y}}{\Green{M}{P}{x}{y}}\dz\right\}=0.
\end{equation}

\vskip 2mm

{\em (ii)} $V$ is a {\em  semismall perturbation} of
$P$ in $M$ if for some $x_0\in M$ we have
 \be \label{semisperteq} \lim_{j\rightarrow
\infty}\left\{\sup_{y\in M_{j}^*} \int_{M_{j}^*}
\frac{\Green{M}{P}{x_0}{z}|V(z)|\Green{M}{P}{z}{y}}
{\Green{M}{P}{x_0}{y}}\dz\right\}=0. \end{equation}
 }
 \end{definition}

\vskip 3mm

 Recall that small perturbations are semismall \cite{Msemismall}. For semismall perturbations we have
\begin{theorem}[\cite{Msemismall,P89,P90}]\label{thmssp}
Let~$P$ be a subcritical operator in~$M$.
Assume that $V=V_+-V_-$
is a semismall perturbation of~$P^*$ in~$M$
satisfying~$V_- \not=0$, where  $V_\pm(x)=\max\{0, \pm V(x)\}$.

 Then there exists $\alpha_0>0$ such that $P_\alpha:= P+\alpha V$
 is subcritical  in $M$ for all $0\leq \alpha< \alpha_0$
 and critical for $\alpha=\alpha_0$.

 Moreover, let $\varphi$ be the
ground state of $P+\alpha_0 V$ and let $y_0$ be a fixed reference point in $M_1$.
Then  for any $0\leq \alpha< \alpha_0$
$$   \varphi   \asymp   \Green{M}{P_\alpha}{\cdot}{y_0} \qquad \mbox {in } M_1^*,
$$
where the equivalence constant depends on $\ga$.
\end{theorem}
%

 \mysection{Capacitory potential and heat content}\label{secaux}
Our first result concerning the large time behavior of positive solutions is given by the following simple lemma that does not distinguish between null-critical and positive-critical operators.
\label{Auxiliary results}
\begin{lemma}\label{lem3}
 Assume that $P\mathbf{1}=0$ and that $P$ is critical in $M$. Let
$B:=B(x_0,\delta)\subset\subset M$ be the ball of radius $\delta$
centered at $x_0$, and suppose that $B^*=M\setminus \mbox{cl}(B)$
is connected. Let $w$ be {\bf the heat content of $B^*$}, i.e. the
minimal nonnegative solution of the following initial-boundary
value problem
 \bea\label{eqw3}
 Lu&=&0 \quad \mbox{ in } B^*\times (0,\infty),\nonumber\\
  u&=&0 \quad \mbox{ on  } \partial B\times (0,\infty),\\
  u&=&1  \quad \mbox{ on  }  B^*\times \{0\}.\nonumber
  \eea
Then $w$ is a decreasing function of $t$, and
$\lim_{t\to\infty}w(x,t)=0$ locally uniformly in $B^*$.
\end{lemma}
\begin{proof} Clearly,
\begin{equation}\label{eqw4}
w(x,t)=\int_{B^*} k_P^{B^*}(x,y,t)\dy<\int_{M}
k_P^{M}(x,y,t)\dy=1 .
\end{equation}
 It follows that $0<w<1$ in $B^*\times (0,\infty)$. Let
 $\varepsilon>0$.
By the semigroup identity and (\ref{eqw4}),
 \bea\label{eqw5}
w(x,t+\varepsilon)=\int_{B^*}
k_P^{B^*}(x,y,t+\varepsilon)\dy&=&\nonumber\\[2mm] \int_{B^*}
\left(\int_{B^*}k_P^{B^*}(x,z,t)k_P^{B^*}(z,y,\varepsilon)\dz\right)
\dy&=&\\[2mm]
\int_{B^*}k_P^{B^*}(x,z,t)\left(\int_{B^*}k_P^{B^*}(z,y,\varepsilon)\dy
\right)\dz&<&\\[2mm]
\int_{B^*}k_P^{B^*}(x,z,t)\dz =w(x,t).\nonumber
\eea Hence, $w$ is a decreasing function of $t$, and therefore,
$\lim_{t\to\infty}w(x,t)$ exists. We denote
$v(x):=\lim_{t\to\infty}w(x,t)$. Note that the above argument shows that even in the subcritical case $w$ is a decreasing function of $t$.

For $\tau>0$, consider the function $\nu(x,t;\tau):=w(x,t+\tau)$,
where $t>-\tau$. Then $\nu(x,t;\tau)$ is a nonnegative solution of
the parabolic equation $Lu=0$ in $B^*\times (-\tau,\infty)$ which
satisfies $u=0$ on $\partial B\times (-\tau,\infty)$. By a
standard parabolic argument as $\tau\to\infty$, any converging
subsequence of this set of solutions converges locally uniformly
to a solution of the parabolic equation $Lu=0$ in $B^*\times
\mathbb{R}$ which satisfies $u=0$  on $\partial B\times
\mathbb{R}$. Since
$\lim_{\tau\to\infty}\nu(x,t;\tau)=\lim_{\tau\to\infty}w(x,\tau)=v(x)$,
the limit does not depend on $t$, and $v$ is a solution of the
elliptic equation $Pu=0$ in $B^*$, and satisfies $v=0$  on
$\partial B$. Furthermore, $0\leq v<\mathbf{1}$.

Therefore, $\mathbf{1}-v$ is a positive solution of the equation
$Pu=0$ in $B^*$ which satisfies $u=1$  on $\partial B$. On the
other hand, it follows from the criticality assumption that
$\mathbf{1}$ is the minimal positive solution of the equation
$Pu=0$ in $B^*$ which satisfies $u=1$  on $\partial B$. Thus,
$\mathbf{1}\leq \mathbf{1}-v$, and therefore, $v=0$.
 \end{proof}
\begin{definition}\label{defcapac}
{\em Let $B:=B(x_0,\delta)\subset\subset M$. Suppose that
$B^*=M\setminus \mbox{cl}(B)$ is connected. The nonnegative
(minimal) solution $$v(x,t)=\mathbf{1}-\int_{B^*}
k_P^{B^*}(x,y,t)\dy$$ is called the {\em parabolic capacitory
potential of $B^*$}. Note that $v$ is indeed the minimal
nonnegative solution of the initial-boundary value problem
 \bea\label{fet}
 Lu&=&0 \quad \mbox{ in } B^*\times (0,\infty),\nonumber\\
  u&=&1 \quad \mbox{ on  } \partial B\times (0,\infty),\\
  u&=&0  \quad \mbox{ on  }  B^*\times \{0\}.\nonumber
  \eea
 }\end{definition}

\begin{corollary}\label{lem2}
Under the assumptions of Lemma \ref{lem3},  the parabolic
capacitory potential $v$ of $B^*$  is an increasing function of
$t$, and we have $\lim_{t\to\infty}v(x,t)=1$ locally uniformly in $B^*$.
\end{corollary}
\begin{proof} Using an exhaustion argument it is easily verified that
\begin{equation}\label{eqv1}
v(x,t)=\mathbf{1}-\int_{B^*}
k_P^{B^*}(x,y,t)\dy=\mathbf{1}-w(x,t),
\end{equation}
where $w$ is the heat content of $B^*$. Therefore, the corollary
follows directly from Lemma \ref{lem3}.\end{proof}
\mysection{Varadhan's lemma}\label{secVa}
Varadhan's celebrated lemma (see,
\cite[Lemma~9, p.~259]{Va} or \cite[pp.~192--193]{Pinsky}) deals with the large time behavior of minimal solutions of the Cauchy problem with {\em bounded} initial data (assuming that $P\mathbf{1}=0$). It turns out that the limit as $t\to\infty$ of such solutions might not exist, but, under further conditions, the spacial oscillation of the solution  tends to zero as $t\to\infty$. In the present section we slightly extend Varadhan's lemma (Lemma~\ref{lemVa}). This extended version of the lemma is crucially used in the proof of the null-critical case in Theorem~\ref{thm1}.

Varadhan proved his lemma for {\em positive-critical} operators on
$\mathbb{R}^d$ using a purely probabilistic approach (ours is purely analytic). Our key
observation is that the assertion of Varadhan's lemma is valid in our general setting
under the weaker assumption that the skew product operator
$\bar{P}:= P\otimes I+I\otimes P$ is critical in $\bar{M}:=M\times
M$, where $I$ is the identity operator on $M$. Note that if $\bar{P}$ is critical in $\bar{M}$, then
$P$ is critical in $M$. On the other hand, if $P$ is positive-critical in
$M$, then $\bar{P}$ is positive-critical in $\bar{M}$. Moreover, if
$\bar{P}$ is subcritical in $\bar{M}$, then by part {\em (i)} of Theorem \ref{thm1},
the heat kernel of $\bar{P}$ on $\bar{M}$ tends to zero as
$t\to\infty$. Since the heat kernel of $\bar{P}$ is equal to the
product of the heat kernels of its factors (see Lemma~\ref{lem_hkp}), it follows that if
$\bar{P}$ is subcritical in $\bar{M}$, then $\lim_{t\to\infty}
k_P^{M}(x,y,t)=0$.

Consider the Riemannian product manifold $\bar{M}=M\times M$. A point in
$\bar{M}$ is denoted by $\bar{x}=(x_1,x_2)$. By $P_{x_i}$,
$i=1,2$, we denote the operator $P$ in the variable $x_i$. So,
$\bar{P}=P_{x_1}+P_{x_2}$ is in fact the above skew product operator defined on
$\bar{M}$. We denote by $\bar{L}$ the corresponding parabolic
operator.
 \begin{lemma}[Varadhan's lemma \cite{P03}]\label{lemVa} Assume that $P\mathbf{1}=0$.
Suppose further that $\bar{P}$ is critical on $\bar{M}$. Let $f$
be a continuous bounded function on $M$, and  let
 $$u(x,t)=\int_{M} k_P^{M}(x,y,t)f(y)\dy$$
 be the minimal solution of the Cauchy problem with initial data $f$ on
$M$.  Fix $K\subset \subset M$. Then
 $$\lim_{t\to\infty}\sup_{x_1,x_2\in K}|u(x_1,t)-u(x_2,t)|=0.$$
\end{lemma}
\begin{proof} Denote by $\bar{u}(\bar{x},t):=u(x_1,t)-u(x_2,t)$. Recall that
the heat kernel $\bar{k}(\bar{x},\bar{y},t)$ of the operator
$\bar{L}$ on $\bar{M}$ satisfies
\begin{equation}\label{eqkkk}
\bar{k}_{\bar{P}}^{\bar{M}}(\bar{x},\bar{y},t)=k_P^M(x_1,y_1,t)k_P^M(x_2,y_2,t).
\end{equation}
By (\ref{eqinvar}) and (\ref{eqkkk}), we have
 \bean
\bar{u}(\bar{x},t)=u(x_1,t)-u(x_2,t)=\\
\int_{M}k_P^M(x_1,y_1,t)f(y_1)\dy_1-\int_{M}k_P^M(x_2,y_2,t)f(y_2)\dy_2=\\
\int_{M}\int_{M}k_P^M(x_1,y_1,t)k_P^M(x_2,y_2,t)(f(y_1)-f(y_2))\dy_1\dy_2=\\
\int_{\bar{M}}\bar{k}_{\bar{P}}^{\bar{M}}
(\bar{x},\bar{y},t)(f(y_1)-f(y_2))\,\mathrm{d}\bar{y}.
 \eean
Hence, $\bar{u}$ is the minimal solution of the Cauchy problem for
the equation $\bar{L}\bar{w}=0$ on $\bar{M}\times (0,\infty)$ with the bounded initial data $\bar{f}(\bar{x}):=f(x_1)-f(x_2)$, where $\bar{x}\in \bar{M}$.

Fix a compact set $K\subset\subset M$ and $x_0\in M\setminus K$,
and let $\varepsilon>0$. Let $B:=B((x_0,x_0),\delta)\subset\subset
\bar{M}\setminus \bar{K}$, where $\bar{K}=K\times K$, and $\delta$
will be determined below. We may assume that $B^*=\bar{M}\setminus
\mbox{cl}(B)$ is connected. Then $\bar{u}$ is a minimal solution
of the following initial-boundary value problem
 \bea\label{ibvpf}
 \bar{L}\bar{u}&=&0 \qquad \qquad\qquad\qquad \mbox{ in }  B^*\times (0,\infty),\nonumber\\
  \bar{u}(\bar{x},t)&=&u(x_1,t)-u(x_2,t) \quad
  \mbox{ on } \partial B\times (0,\infty),\\
  \bar{u}(\bar{x},0)&=&f(x_1)-f(x_2)  \qquad\quad  \mbox{ on  }  B^*\times \{0\}.\nonumber
  \eea
We need to prove that $\lim_{t\to\infty}\bar{u}(\bar{x},t)=0$.

By the superposition principle (which obviously holds for minimal
solutions), we have
 $$\bar{u}(\bar{x},t)=u_1(\bar{x},t)+u_2(\bar{x},t)\quad \mbox{ on } B^*\times
[1,\infty),$$ where $u_1$ solves the initial-boundary value
problem
 \bea\label{ibvpu1}
 \bar{L}u_1&=&0 \qquad\qquad\qquad\qquad \mbox{ in } B^*\times (1,\infty),\nonumber\\
  u_1(\bar{x},t)&=&u(x_1,t)-u(x_2,t) \quad \mbox{ on  } \partial B\times (1,\infty),\\
  u_1(\bar{x},0)&=&0  \qquad\qquad\qquad\qquad \mbox{ on  }  B^*\times \{1\},\nonumber
 \eea
 and $u_2$ solves the  initial-boundary value problem
 \bea\label{ibvpu2}
 \bar{L}u_2&=&0 \qquad\qquad\qquad\qquad \mbox{ in } B^*\times (1,\infty),\nonumber\\
  u_2(\bar{x},t)&=&0 \qquad\qquad\qquad\qquad \mbox{ on  } \partial B\times (1,\infty),\\
  u_2(\bar{x},0)&=& u(x_1,1)-u(x_2,1)  \quad \mbox{ on  }  B^*\times \{1\}.\nonumber
  \eea
Clearly, $|\bar{u}(\bar{x},t)|\leq 2\|f\|_\infty$ on
$\bar{M}\times (0,\infty)$. Note that if $\bar{x}=(x_1,x_2)\in
\partial B$, then on $M$, $\mbox{dist}_M(x_1,x_2) \leq 2\delta$. Using Schauder's
parabolic interior estimates on $M$, it follows that if $\delta$
is small enough, then
$$|\bar{u}(\bar{x},t)|=|u(x_1,t)-u(x_2,t)|<\varepsilon \quad
\mbox{ on  }
\partial B\times (1,\infty).$$
By comparison of $u_1$ with the parabolic capacitory potential of
$B^*$, we obtain that
\begin{equation}\label{estu1}
|u_1(\bar{x},t)|\leq \varepsilon\left(1-\int_{B^*}
\bar{k}_{\bar{P}}^{B^*}(\bar{x},\bar{y},t-1)\,
\mathrm{d}\bar{y}\right)<\varepsilon \qquad \mbox{ in } B^*\times
(1,\infty).
  \end{equation}
On the other hand,
\begin{equation}\label{estu2}
|u_2(\bar{x},t)|\leq 2\|f\|_\infty\int_{B^*}
\bar{k}_{\bar{P}}^{B^*}(\bar{x},\bar{y},t-1)\,\mathrm{d}\bar{y} \qquad
\mbox{ in } B^*\times (1,\infty).
  \end{equation}
It follows from (\ref{estu2}) and Lemma \ref{lem3} that there
exists $T>0$ such that
\begin{equation}\label{estu21}
|u_2(\bar{x},t)|\leq \varepsilon \quad \mbox { for all }
\bar{x}\in \bar{K} \mbox{ and } t>T.
\end{equation}
Combining (\ref{estu1}) and (\ref{estu21}), we obtain that
$|u(x_1,t)-u(x_2,t)|\leq 2\varepsilon$ for all $x_1,x_2\in K $ and
$t>T$. Since $\varepsilon$ is arbitrary, the lemma is proved. \end{proof}

\mysection{Existence of $\lim_{t\to\infty} e^{\lambda_0 t}k_P^{M}(x,y,t)$}\label{secmainthm}
The present section is devoted to the proof of Theorem~\ref{thm1} which claims that
$\lim_{t\to\infty} e^{\lambda_0 t}k_P^{M}(x,y,t)$ always
exists. Without loss of generality, we may assume that $\gl_0=0$. For convenience, we denote by $k$ the heat kernel
$k_P^{M}$ of the operator $P$ in $M$.
\begin{proof}[Proof of Theorem~\ref{thm1}]

(i) {\em The subcritical case}: Suppose that $P$ is subcritical in $\Gw$. It means that for any $x,y,\in M$ we have
\begin{equation}\label{JFA92_3.22}
 \int_1^\infty k(x, y, t) \dt < \infty.
\end{equation}
Let $x, y\in\Gw$ be fixed, and suppose that $k(x, y, t)$ does not converge to zero as
$t \to \infty$. Then there exist an increasing sequence $t_j\to\infty$, $t_{j+ 1} - t_j > 1$, and
$\vge > 0$, such that $k(x, y, t_j) > \vge$ for all $j > 1$. Using the parabolic Harnack
inequality \eqref{eq_harnack}, we deduce that there exists $C > 0$ ($C$ may depend on $x, y,\Gw$,
and $P$) such that
\begin{equation}\label{JFE92_3.23}
\int_{t_j}^{t_{j+1}}k(x, y, t) \dt > C\vge\qquad  \forall j>2.
\end{equation}
But this contradicts \eqref{JFA92_3.22}.

\vskip 3mm

(ii) {\em The positive-critical case}: Let $P$ be a critical operator in $M$ and let $\{t_j\}_{j=1}^\infty\subset \R$ be a sequence such that $t_j\to\infty$, and define:
$$u_j(x, y, t)=k(x, y, t+t_j).$$
Fix $x_0,y_0\in M$. By \eqref{eqHar1} we have
\begin{equation}\label{eqHar2}
u_j(x,y_0,t)\leq c_1(y_0)\vgf(x), \qquad \Big( u_j(x_0,y,t)\leq
c_2(x_0)\vgf^*(y), \Big)
\end{equation}
for all $x\in M$ ($y\in M$) and $t\in \R$ and $j>J(t)$. Using the parabolic Harnack inequality and a standard Parabolic regularity argument we may subtract a subsequence of $\{u_j\}$ (which we rename by $\{u_j\}$) that converges locally uniformly to a nonnegative solution $u(x, y, t)$ of the parabolic equations
$$Lu=\pd_t u+P(x,\pd_x)u=0, \quad Lu=\pd_t u+P(y,\pd_y)u=0 \quad   \mbox{ in } M\times \R.$$
Moreover, $u$ satisfies the estimates
\begin{align}\label{eqHar3}
u(x,y_0,t)&\leq c_1(y_0)\vgf(x) \qquad \forall(x,t)\in M\times \R, \\
u(x_0,y,t)&\leq
c_2(x_0)\vgf^*(y) \qquad \forall(y,t)\in M\times \R. \label{eqHar3a}
\end{align}

 Using the semigroup property
we have for all $t, \gt > 0$
\begin{align}\label{JFA92_3.35}
\int_M k(x,z,\gt)k(z,y,t + t_j) \dz= \int_M k (x,z ,t + t_j) k(z,y, \gt)\dz\\
=k(x, y, \gt + t + t_j)
    \end{align}
for all $j\geq 1$. On the other hand, by \eqref{eqinvar}, the ground state $\vgf$ is a positive invariant solution. Consequently, for any $x\in M$  and $\gt >0$, $k(x,\cdot,\gt)\vgf\in L^1( M)$.
Similarly, for all $\gt > 0$ and $y\in M$, $\vgf^* k(\cdot, y, \gt)\in L^1( M)$. Hence, by estimates \eqref{eqHar2}
and the Lebesgue dominated convergence theorem, we obtain
\begin{equation}\label{JFA92_3.36}
\int_ M k(x, z, \gt) u(z, y, t) \dz =
\int_ M u(x, z, t) k(z,y, \gt) \dz = u(x, y, \gt + t),
\end{equation}
where $u(x, y, t) = \lim_{j\to\infty} k(x, y,t + t_j)$. In particular, for $\gt = t_j$ and $t = - t_j$ we have
\begin{equation}\label{JFA92_3.37}
\int_ M k(x, z, t_j) u(z, y,-t_j) \dz
=\int_ M u(x, z,-t_j) k(z,y, t_j) \dz = u(x, y, 0)
\end{equation}
Invoking again estimates \eqref{eqHar1}, and \eqref{eqHar3}--\eqref{eqHar3a}, we see that the integrands in \eqref{JFA92_3.37} are bounded by
$C(x,y)\vgf(z)\vgf^*(z)$. Recall our assumption that $\vgf(x)\vgf^*(x)\in L^1( M)$, hence, the Lebesgue dominated convergence theorem implies
\begin{equation}\label{JFA92_3.38}
\int_ M u(x, z, 0) \hat{u}(z, y,0) \dz =
\int_ M \hat{u}(x, z,0) u(z,y,0) \dz = u(x, y,0),
    \end{equation}
where
\begin{equation}\label{JFA92_3.39}
\hat{u}(x, y,t):= \lim_{j\to\infty}u(x, y,t-t_j),
\end{equation}
and again we may assume that the limit in (\ref{JFA92_3.39}) exists.

On the other hand, by the invariance property  \eqref{eqinvar}, and  estimates \eqref{eqHar1}, we get
\begin{equation}\label{JFA92_3.40}
\int_ M u(x, z, \gt) \vgf(z) \dz = \vgf(x), \qquad
\int_ M \vgf^*(z) u(z,y, t) \dz = \vgf^*(y).
\end{equation}
In particular $u(x, y, t) > 0$. It follows from the Harnack inequality and
(\ref{JFA92_3.38}) that $\hat{u}(x, y, t)$ is also positive and we have
\begin{equation}\label{JFA92_3.41}
\int_M\int_ M \hat{u}(x, z, 0) u(z, y,0) \dz\vgf(y) \dy =
\int_ M u(x, y,0)\vgf(y) \dy .
\end{equation}
Consequently, (\ref{JFA92_3.40}) and (\ref{JFA92_3.41}) imply that
\begin{equation}\label{JFA92_3.42}
    \int_M \hat{u}(x, z, 0) \vgf(z) \dz=\vgf(x).
\end{equation}
Define integral operators
\begin{equation}\label{JFA92_3.43}
    Uf(x):=\int_M \hat{u}(x, z, 0)f(z) \dz, \qquad U^*f(y):=\int_ M \hat{u}(z, y, 0)f(z) \dz.
\end{equation}
By (\ref{JFA92_3.38}) and (\ref{JFA92_3.42}) we see that $\vgf(x)$ and $u(x, y, 0)$ (as a function of $x$) are
positive eigenfunctions of the operator $U$ with an eigenvalue $1$, and for
every $x \in M$, $u(x, y, 0)$ is a positive eigenfunction of $U^*$ with an eigenvalue
$1$. Moreover, for every $x\in  M$, $u(x, \cdot, 0)\vgf\in L^1( M)$ and for every $x,
z\in M$, $u(x, \cdot, 0) u(\cdot, z, 0) \in L^1( M)$. Consequently, it follows \cite[Lemma~3.4]{Pheat} that $1$ is a simple eigenvalue of the integral operator $U$. Hence,
\begin{equation}\label{JFA92_3.44}
u(x, y,0) = \gb(y)\vgf(x).
\end{equation}
By  (\ref{JFA92_3.40}) and (\ref{JFA92_3.44}) we have
\begin{equation}\label{JFA92_3.45}
\int_ M \vgf^* (z)\gb(y)\vgf(z)\dz =\int_ M \vgf^* (z)u(z, y,0)\dz
=\vgf^* (y) .
\end{equation}
Therefore, we obtain from \eqref{JFA92_3.44} and \eqref{JFA92_3.45} that $\gb(y)= \frac{\vgf^*(y)}{\int_ M \vgf(z)\vgf^*(z)\dz}$.

Thus,
\begin{equation}\label{JFA92_3.46}
\lim_{j\to\infty} k(x, y, t_j) = u(x, y,0)=\frac{\vgf(x)\vgf^*(y)}{\int_ M \vgf(z)\vgf^*(z)\dz}\, ,
\end{equation}
and the positive-critical case is proved since the limit in \eqref{JFA92_3.46} is independent of the sequence $\{t_j\}$.

\vskip 3mm

(iii) {\em The null-critical case}:
Without loss of generality, we may assume that $P\mathbf{1}=0$,
where $P$ is a null-critical operator in $M$. We need to prove
that
 $$\lim_{t\to\infty} k_P^M(x,y,t)=0.$$

As in Lemma~\ref{lemVa} (Varadhan's lemma), consider the Riemannian product manifold $\bar{M}:=M\times
M$, and let $\bar{P}=P_{x_1}+P_{x_2}$ be the corresponding skew
product operator which is defined on $\bar{M}$.

If $\bar{P}$ is subcritical on $\bar{M}$, then by part {\em (i)},  $\lim_{t\to \infty}
\bar{k}_{\bar{P}}^{\bar{M}}(\bar{x},\bar{y},t)=0$, and since
$$\bar{k}_{\bar{P}}^{\bar{M}}(\bar{x},\bar{y},t)=
k_{P}^M(x_1,y_1,t)k_{P}^M(x_2,y_2,t),$$ it follows that
$\lim_{t\to \infty}k_P^M(x,y,t)=0$.

Therefore, there remains to prove the theorem for the case where
$\bar{P}$ is critical in $\bar{M}$. Fix a nonnegative, bounded,
continuous function $f\neq 0$ such that $\varphi^*f\in L^1(M)$,
and consider the solution
 $$v(x,t):=\int_M k_P^M(x,y,t)f(y)\dy.$$
Let $t_n\to \infty$, and consider the sequence $\{v_n(x,t):= v(x,t+t_n)\}$. As in part {\em (ii)}, up to a subsequence,  $\{v_n\}$ converges to a nonnegative solution $u\in \mathcal{H}_P(M\times
\mathbb{R})$.

Invoking Lemma \ref{lemVa} (Varadhan's lemma), we see that
$u(x,t)=\alpha(t)$. Since $u$ solves the parabolic equation
$Lu=0$, it follows that $\alpha(t)$ is a nonnegative constant
$\alpha$.

We claim that $\alpha=0$. Suppose to the contrary that $\alpha>0$.
The assumption that $\varphi^*f\in L^1(M)$ and (\ref{eqinvar})
imply that for any $t>0$
 \bea\label{eqphiv} \int_M
\varphi^*(y)v(y,t)\dy=\int_M \varphi^*(y)\left(\int_M
k_P^M(y,z,t)f(z)\dz\right)\dy=\nonumber\\ \int_M \left(\int_M
\varphi^*(y) k_P^M(y,z,t)\dy\right)f(z)\dz=\int_M \varphi^*(z)
f(z)\dz<\infty . \eea
 On the other hand, by the null-criticality, Fatou's lemma, and
(\ref{eqphiv}) we have
  \bean\label{eqcontra}
  \infty=\int_M \varphi^*(z) \alpha\dz=
   \int_M \varphi^*(z) \lim_{n\to\infty}v(z,t_n)\dz
\leq\\
 \liminf_{n\to\infty}\int_M \varphi^*(z) v(z,t_n)\dz=\int_M
\varphi^*(z) f(z)\dz<\infty. \eean Hence $\alpha=0$, and
therefore
\begin{equation}\label{eqlim0}
\lim_{t\to\infty}\int_M
k_P^M(x,y,t)f(y)\dy=\lim_{t\to\infty}v(x,t)=0.
\end{equation}

Now fix $y\in M$ and let $f:=k_P^M(\cdot,y,1)$. Consider the minimal solution of the Cauchy problem with
initial data $f$. So, by the semigroup property we have
\begin{multline*}u(x,t):=\int_M
k_P^M(x,z,t)f(z)\dz=\\ \int_M
k_P^M(x,z,t)k_P^M(z,y,1)\dz =k_P^M(x,y,t+1).
\end{multline*}
In view of  \eqref{eqHar1} (with $v=\mathbf{1}$) and \eqref{eqinvar}, the function $f$ is bounded and satisfies $f\vgf^*\in L^1(M)$. Therefore, by \eqref{eqlim0}
$\lim_{t\to\infty}u(x,t)=0$. Thus,
 $$\lim_{t\to\infty}k_P^M(x,y,t)=\lim_{t\to\infty}k_P^M(x,y,t+1)= \lim_{t\to\infty}u(x,t)=0.$$

The last statement of the theorem concerning the behavior of the Green function $G_{P-\gl}^M(x,y)$ as $\gl\to\gl_0$ follows from the first part of the theorem using a classical Abelian theorem \cite[Theorem~10.2]{Simon}.
 \end{proof}
\mysection{Applications of Theorem~\ref{thm1}}\label{sec_appl}
We discuss in this section some applications of Theorem~\ref{thm1} and comment on related results. First we prove \eqref{logform},  which characterizes  $\lambda_0$ in terms of the large time behavior of $\log k_P^M(x,y,t)$.
\begin{Cor}[\cite{fkp}]\label{asser_log}
The heat kernel satisfies
  \begin{equation}\label{eq_logform1}
\lim_{t \to \infty} \frac{\log k_P^M(x,y,t)}{t} =-\lambda_0(P,M) \qquad x,y\in M.
\end{equation}
\end{Cor}
\begin{proof}
The needed upper bound for the validity of \eqref{eq_logform1}
follows directly from Theorem~\ref{thm1} and \eqref{eq_hkp-gl}.

The lower bound is obtained by a standard exhaustion argument and Theorem~\ref{thm1} (cf.~the proof of  \cite[Theorem~10.24]{Grigoryan}). Indeed, let $\{M_{j}\}_{j=1}^{\infty}$ be an {\em exhaustion} of  $M$. Recall that since $M_j$ is a smooth bounded domain, the operator $P-\gl_0(P,M_j)$ is positive-critical in $M_j$.
Therefore, Theorem~\ref{thm1}, and the monotonicity of heat kernels with respect to domains (see part 4 of Lemma~\ref{lem_hkp}) imply that
$$\liminf_{t \to \infty} \frac{\log k_P^M(x,y,t)}{t}\geq \lim_{t \to \infty} \frac{\log k_P^{M_j}(x,y,t)}{t}=-\lambda_0(P,M_j).$$
Since $\lim_{j\to\infty}\gl_0(P,M_j)= \gl_0(P,M)$, we obtain the needed lower bound.
\end{proof}

\vskip 3mm

 We now use Theorem~\ref{thm1} to strengthen Lemma
\ref{lemVa}. More precisely, in the following result we obtain the large time behavior of solutions of the Cauchy problem with initial conditions which satisfy a certain (and in some sense optimal) integrability condition.
 \begin{corollary}[\cite{P03}]\label{Cormainthm} Let $P$ be an elliptic operator of
the form (\ref{P}) such that $\lambda_0\geq 0$. Let $f$ be a
continuous function on $M$ such that $v^*f\in L^1(M)$ for some
$v^*\in \mathcal{C}_{P^*}(M)$. Let
 $$u(x,t)=\int_{M} k_P^{M}(x,y,t)f(y)\dy$$
 be the minimal solution of the Cauchy problem with initial data $f$ on
$M$. Fix $K\subset\subset M$.   Then $$\lim_{t\to\infty}\sup_{x\in
K}|u(x,t)-\mathcal{F}(x)|=0,$$ where $$\mathcal{F}(x):=
  \begin{cases}
    \dfrac{\int_M f(y)\varphi^*(y)\dy}{\int_M
\varphi(y)\varphi^*(y)\dy}\,\varphi(x) & \;\text{{\em if $P$ is
positive-critical in $M$}},
\\[4mm]
    0 & \;\text{{\em otherwise}}.
  \end{cases}
$$
\end{corollary}
\begin{proof}
Since $v^*f\in L^1(M)$,  estimate  (\ref{eqHar1}),
and the dominated convergence theorem imply that
$$\lim_{t\to\infty} u(x,t)= \lim_{t\to\infty} \int_{M} k_P^{M}(x,y,t)f(y)\dy= \int_{M}\lim_{t\to\infty}\big[ k_P^{M}(x,y,t)f(y)\big]\dy,$$
and the claim of the corollary follows from Theorem~\ref{thm1}.
\end{proof}

\vskip 3mm

Assume now that $P\mathbf{1}=0$ and $\int_M k_P^M(\cdot,y,t)\dy=\mathbf{1}$ (i.e. $\mathbf{1}$ is a positive invariant solution
of the operator $P$ in $M$). Corollary \ref{Cormainthm} implies
that for any $j\geq 1$ and all $x\in M$ we have
\begin{equation}\label{eqvMstar}
\lim_{t\to\infty}\int_{M_j^*}\!\! k_P^{M}(x,y,t)\dy\!=\!
  \begin{cases}
    \dfrac{\int_{M_j^*} \varphi^*(y)\dy}{\int_M
\varphi^*(y)\dy} & \!\text{if $P$ is positive-critical in $M$},
\\[4mm]
    1 & \!\text{otherwise}.
  \end{cases}\nonumber
\end{equation}
Suppose further that $P$ is not positive-critical in $M$, and $f$ is a
bounded continuous function such that $\liminf_{x\to\infty}
f(x)=\varepsilon>0$.  Then
\begin{equation}\label{eqliminf}
  \liminf_{t\to\infty}\int_{M} k_P^{M}(x,y,t)f(y)\dy\geq \varepsilon.
\end{equation}
 Hence, if the integrability condition of Corollary \ref{Cormainthm}
is not satisfied, then the large time behavior of the minimal
solution of the Cauchy problem may be complicated. The following
example of W.~Kirsch and B.~Simon \cite{KS} demonstrates this
phenomenon.
\begin{example}\label{exKS}{\em
Consider the heat equation in $\mathbb{R}^d$. Let $R_j:=\mathrm{e}^{\mathrm{e}^j}$
and let
 $$ f(x):=2+(-1)^j \qquad  \mbox{ if } \quad R_j<\sup_{1\leq i\leq d}
|x_i|<R_{j+1}, \;j\geq 1.$$
 Let $u$ be the minimal solution of the Cauchy problem with  initial data $f$.
Then for $t\sim R_jR_{j+1}$ one has that $u(0,t)\sim 2+(-1)^j$,
and thus $u(0,t)$ does not have a limit. Note that by Lemma
\ref{lemVa}, for $d=1$, $u(x,t)$ has exactly the same asymptotic
behavior as $u(0,t)$ for all $x\in \mathbb{R}$.

In fact, it was proved in \cite{RE} that for the heat equation on $\R^d$, the following holds: for any bounded function $f$ defined on $\R^d$, the limit
$$\lim_{t\to\infty}\int_{\R^d} k_{-\Gd}^{\R^d}(x,y,t)f(y)\dy$$
exists, if and only if the limit
$$\lim_{R\to\infty} \frac{1}{|B(x,R)|}\int_{B(x,R)}f(y) \dy $$ exists.
Moreover, the values of the two limits are equal. For an extension of the above result see \cite{PLi86}.
We note that such a theorem is false if the average value of $f$ is taken on solid cubes.

 }\end{example}
 The next three corollaries concern elliptic operators on manifolds.
 \begin{corollary}[\cite{CK}]\label{cor_l2}
    Let $M$ be a noncompact complete Riemannian manifold, and denote by $\Gd$ the corresponding Laplace-Beltrami operator. Then
    \begin{equation}\label{eq_l2}
       \ell_0:=\lim_{t\to\infty}k_{-\Delta}^{M}(x,y,t)=0
    \end{equation}
    if and only if $M$ has an infinite volume.
 \end{corollary}
 \begin{proof}
If $M$ has a finite volume, then $-\Gd$ is positive-critical in $M$, and by Theorem~\ref{thm1}, we have $\ell_0=\big(\mathrm{vol}(M)\big)^{-1}>0$.

Suppose now that $M$ has an infinite volume. If $\gl_0>0$, Corollary~\ref{asser_log} implies that $\ell_0=0$. On the other hand, if $\gl_0=0$, then since $u=\mathbf{1}$ is a positive harmonic function which is not in $L^2(M)$, the uniqueness of the ground state implies that $-\Gd$ is not positive-critical in $M$. Hence, by Theorem~\ref{thm1} we have $\ell_0=0$.
 \end{proof}
Before studying the next two results we recall three basic notions from the theory of manifolds with group actions.
 \begin{definition}\label{degG}{\em
Let $G$ be a group, and suppose that $G$ acts on $M$.  For any real continuous function $v$ and $g\in G$, denote by $v^g$ the function defined by $v^g(x):=v(gx)$. Let $\R^*=\R\setminus \{0\}$ be the real multiplicative group.

1. A nonzero real continuous function $f$ on $M$ is called {\em $G$-multi\-plicative} if there exists a group homomorphism $\gg : G\to \R^*$, such that
$$f(gx)=\gg(g)f(x)\qquad \forall g\in G, x\in M.$$

2. A $G$-group action on $M$ is {\em compactly generating} if there exists $K\Subset M$ such that $GK=M$ (see \cite{LinP}).

3.  The operator $P$ is said to be $G$-equivariant if
$$P[u^g]=(P[u])^g\qquad \forall g\in G.$$
   }
 \end{definition}

\begin{corollary}[cf. \cite{CK}]\label{cor_covering}
Suppose that a noncompact manifold $M$ is a covering of a compact Riemannian
manifold, and consider the Laplace-Beltrami operator $\Gd$ on $M$.  Then
    \begin{equation}\label{eq_l2a}
       \ell :=\lim_{t\to\infty}\mathrm{e}^{\gl_0 t} k_{-\Delta}^{M}(x,y,t)=0.
    \end{equation}
 In particular, $-\Gd-\gl_0$ is not positive-critical in $M$.
     \end{corollary}
 \begin{proof}
Suppose that $\ell> 0$. By Theorem~\ref{thm1}, the operator $-\Gd-\gl_0$ is positive-critical in $M$, and in particular, the ground state $\varphi$ is in $L^2(M)$. The uniqueness of the ground state implies that $\varphi$ is $G$-multiplicative, where $G$ is the deck transformation.
But this contradicts the assumption that  $\varphi\in L^2$, since $G$ is an infinite Group.
 \end{proof}
The following corollary deals with manifolds with group actions and {\em general} equivariant operators, without any assumption on the growth of the acting groups. It generalizes Corollary~\ref{cor_covering}, and seems to be new.
\begin{corollary}\label{cor_l3}
Let $M$ be a noncompact manifold, and let $G$ be a group. Suppose that $G$ acts on $M$, and that the $G$-group action on $M$ is compactly generating. Assume that $P$ is a $G$-equivariant second-order elliptic operator  of the form \eqref{P} which is defined on $M$. Suppose further that any nonzero $G$-multiplicative $C^{2,\ga}$-function is not integrable on $M$.
 Then
    \begin{equation}\label{eq_l3}
\ell:=\lim_{t\to\infty}\mathrm{e}^{\gl_0 t} k_{P}^{M}(x,y,t)=0.
    \end{equation}
     \end{corollary}
 \begin{proof} Suppose that $\!\ell\!>\! 0$. By Theorem~\ref{thm1},  the operator $P\!-\!\gl_0\!$ is positive-critical in $M$. In particular, $\dim \mathcal{C}_{P\!-\!\gl_0}(M)\!=\!\dim \mathcal{C}_{P^*-\gl_0}(M)=1$. By the uniqueness of the ground state, it follows that the corresponding ground states $\vgf$ and $\vgf^*$ are $G$-multiplicative. Therefore, their product $\varphi\vgf^*$ is also $G$-multiplicative, and we arrived at a contradiction since by our assumption a nonzero $G$-multiplicative $C^{2,\ga}$-function is not integrable on $M$.
 \end{proof}
\begin{remark}\label{rem_linp}{\em

1. Corollary~\ref{cor_l3} clearly applies to the case where $M$ is a regular covering of a compact Riemannian manifold, and $G$ is the corresponding deck transformation. In fact, in this case, by \cite[Theorem~5.17]{LinP},  the product of the ground states  is a positive $G$-invariant function, and therefore it is not integrable on $M$.

2. Corollary~\ref{cor_l3} applies in particular to the case where $P$ is defined on $\R^d$, and $P$ is $\mathbb{Z}^d$-equivariant (that is, $P$ has $\mathbb{Z}^d$-periodic coefficients). In this case it is known that $P-\gl_0$ is (null)-critical if and only if $d=1,2$. For further related results, see \cite{LinP}.
 }
\end{remark}

\vskip 3mm

Finally, let us introduce an important subclass of elliptic operators $P$ such that $P-\gl_0$ is positive-critical in $M$.
\begin{definition}\label{IU}{\em
Assume that $P$ is symmetric and  $P-\gl_0$ is positive-critical in $M$ with a ground state $\varphi$ satisfying the normalization condition $\|\varphi\|_{L^2(M)}=1$. Let $T_P^M(t)$ be the corresponding (Dirichlet) semigroup on $L^2(M)$ generated by $P$.  We say that $T_P^M(t)$ is
{\em intrinsically ultracontractive} if for any $t > 0$ there exists a positive-constant $C_t$ such that
$$0\leq \mathrm{e}^{\gl_0 t} k_{P}^{M}(x,y,t)\leq c_t\varphi(x)\varphi(y)\qquad \forall x,y\in M.$$
 }
\end{definition}
\begin{example}\label{ex_IU}{\em
If $M$ is a smooth bounded domain and $P:=-\Gd$, then the semigroup $T_P^M(t)$ is intrinsically ultracontractive on $L^2(M)$. Also, let $M:=\R^d$ and $P:=-\Gd+(1+|x|^2)^{\ga/2}$, where $\ga\in\R$. Then $T_P^M(t)$ is intrinsically ultracontractive on $L^2(M)$ if and only if $\ga>2$ (see also \cite{Dheat,Mskew} and references therein).
 }
\end{example}
It turns out that in the intrinsically ultracontractive case the rate of the convergence (as $t\to\infty$)  of the $h$-transformed heat kernel is uniformly exponentially fast.
\begin{theorem}[{\cite[Theorem 4.2.5]{Dheat}}]\label{thmIU}
Assume that $P$ is symmetric and subcritical in $M$, and suppose that $T_P^M(t)$ is intrinsically ultracontractive  on $L^2(M)$. Then there exists a complete orthonormal set $\{\varphi_j\}_{j=0}^\infty$ in $L^2(M)$ such that $\varphi_0=\varphi$, and for
any $j\geq 0$ the function $\varphi_j$ is an eigenfunction of the Friedrichs extension of the operator $P$  with eigenvalue $\gl_j$, where $\{\gl_j\}$ is nondecreasing. Moreover, the heat kernel has the eigenfunction expansion \eqref{eq_ef_exp},  and there exist positive-constants $C$ and $\gd$ such that
$$\left|e^{\lambda_0 t} \frac{k_P^{M}(x,y,t)}{\varphi_0(x)\varphi_0(y)}-1\right|\leq Ce^{-\gd t}\qquad \forall x,y\in M, t > 1. $$

\vskip 2mm

\noindent Furthermore, for  any $\vge > 0$ the series
$$e^{\lambda_0 t} \frac{k_P^{M}(x,y,t)}{\varphi_0(x)\varphi_0(y)}=\sum_{n=0}^\infty e^{-\lambda_n t}\frac{\varphi_n(x)\varphi_n(y)}{\varphi_0(x)\varphi_0(y)}$$
converges uniformly on $M\times M\times[\vge,\infty)$.
 \end{theorem}
 \begin{remark}\label{rem_M}{\em
Murata \cite{MJFA07} proved that if $T_P^M(t)$ is intrinsically ultracontractive  on $L^2(M)$, then $\mathbf{1}$ is a small perturbation of $P$ in $M$. In particular, for any $n\geq 0$ the function $\vgf_n/\vgf_0$ is bounded, and has a continuous extension up to the Martin boundary of $M$ (with respect to $P$) \cite{P99}. On the other hand, an example of Ba\~{n}uelos and Davis in \cite{BD} gives
us a finite area domain $M\subset\R^2$ such that $\mathbf{1}$ is a small perturbation of the
Laplacian in $M$, but the corresponding semigroup is not intrinsically ultracontractive.
 }
 \end{remark}

 \begin{remark}\label{remdec}{\em
In the null-recurrent case, the heat kernel may decay very slowly
as $t\to \infty$, and one can construct a complete Riemannian
manifold $M$ such that all its Riemannian products $M^j, j\geq 1$
are null-recurrent with respect to the Laplace-Beltrami operator on $M^j$ (see \cite{CG}).
}
\end{remark}
\mysection{Davies' conjecture concerning strong ratio limit}\label{sect1DC}
Having proved in Section~\ref{secmainthm} that $\lim_{t\to\infty} \mathrm{e}^{\lambda_0
t}k_P^{ M}(x,y,t)$ always exists, we next ask how fast
this limit is approached.  It is natural to conjecture that the
limit is approached equally fast for different points $x,y\in
 M$. Note that in the context of Markov chains, such an
{\em (individual) strong ratio limit property} is in general not
true \cite{Chu}. The following conjecture was raised by
E.~B.~Davies \cite{D} in the selfadjoint case.
\begin{conjecture}[Davies' conjecture]\label{conjD}
Let $Lu=u_t+P(x, \partial_x)u$ be a parabolic operator which is
defined on a noncompact Riemannian manifold $ M$, and assume that $\gl_0(P,M)\geq 0$. Fix a reference
point $x_0\in  M$. Then
\begin{equation}\label{eqconjD}
\lim_{t\to\infty}\frac{k_P^ M(x,y,t)}{k_P^ M(x_0,x_0,t)}=a(x,y)
\end{equation}
exists and is positive for all $x,y\in  M$.
\end{conjecture}
The aim of the present section is to discuss Conjecture~\ref{conjD}
and closely related problems, and to obtain some results under
minimal assumptions. Since, the conjecture does not depend on the value of
$\lambda_0$, we assume throughout the present section that $\lambda_0=0$.

\vskip 3mm

Conjecture~\ref{conjD} was recently disproved by G.~Kozma \cite{Kozma} in the {\em discrete} setting:
\begin{theorem}[Kozma \cite{Kozma}]\label{thm_Kozma}
There exists a connected graph $G$ with bounded weights and two vertices
$x, y \in G$ such that
\begin{equation}\label{Kozma}
\lim_{n\to\infty}\frac{k(x, x, n)}{k(y, y, n)}
\end{equation}
does not exist, where $k$ is the heat kernel of the lazy random walk such that the walker, at every step,
chooses with probability $1/2$ to stay in place, and with probability $1/2$ to move to one
of the neighbors (with probability proportional to the given weights).
\end{theorem}
In addition, Kozma indicates in \cite{Kozma} how the construction
might be carried out in the category of Laplace-Beltrami operators on manifolds with bounded geometry. As noted in \cite{Kozma}, the bounded weights property in the graph's setting is the analogue to the bounded geometry property in the manifold
setting, a property that in fact was not assumed in Davies' conjecture.
We note that in Kozma's example the ratio $\frac{k(x, x, n)}{k(y, y, n)}$ is bounded between two constants that are independent of $t$.

\vskip 3mm

Nevertheless, as we show below, there are many situations where Conjecture~\ref{conjD} holds true.
\begin{rem}\label{6Rem82}{\em
1. Theorem~\ref{thm1} implies that Conjecture~\ref{conjD} holds
true in the positive-critical case, therefore, we assume throughout the present section that $P$ {\em is not positive-critical}.

\vskip 3mm

2. We also note that if $P$ is {\em symmetric}, then the conjecture holds if
\begin{equation}\label{eq_liouv}
\dim \mathcal{C}_{P-\gl_0}(M)=1,
\end{equation}
(i.e. in the ``Liouvillian" case) see \cite[Corollary 2.7]{ABJ}.
In particular, it holds true for {\em critical selfadjoint} operators. It also holds for the Laplace-Beltrami operator on a complete Riemannian manifold of dimension $d$ with
nonnegative Ricci curvature \cite{D}.

\vskip 3mm

3. Recently Agmon \cite{Agp} obtained the exact asymptotics (in $(x,y,t)$) of the heat kernel for a $\mathbb{Z}^d$-periodic
(non-selfadjoint) operator $P$ on $\mathbb{R}^d$, and for an equivariant operator $P$ defined on {\em abelian} cocompact covering manifold $M$. In particular, it follows from
Agmon's results that Conjecture~\ref{conjD} holds true in these cases. Note that in these cases (and even in the case of {\em nilpotent} cocompact covering) it is known \cite{LinP} that
\begin{equation}\label{eqonedP}
\dim \mathcal{C}_{P-\gl_0}( M)=\dim
\mathcal{C}_{P^*-\gl_0}( M)=1.
\end{equation}

\vskip 3mm

4. For other particular cases where the conjecture holds true see \cite{ABJ,CCFI,CMS,CMS1,D,Wong}.
 }
\end{rem}
\begin{rem}\label{6Rem29}{\em
It would be interesting to prove Conjecture~\ref{conjD}
at least under the assumption
\begin{equation}\label{eqoned}
\dim \mathcal{C}_{P-\gl_0}( M)=\dim
\mathcal{C}_{P^*-\gl_0}( M)=1,
\end{equation}
which holds true in the critical case and in many important
subcritical (Liouvillian) cases.  For a probabilistic interpretation of
Conjecture~\ref{conjD}, see \cite{ABJ}.
 }
\end{rem}
\begin{rem}\label{6Rem21}{\em
Let $t_n\to \infty$. By a standard parabolic argument, we may
extract a subsequence $\{t_{n_k}\}$ such that for every $x,y\in
 M$ and $s<0$
\begin{equation}\label{eqsequence}
a(x,y,s):=\lim_{k\to\infty}\frac{k_P^ M(x,y,s+t_{n_k})}{k_P^ M(x_0,y_0,t_{n_k})}
 \end{equation}
 exists. Moreover, $a(\cdot,y,\cdot)\in
\mathcal{H}_P( M\times \mathbb{R}_-)$, and $a(x,\cdot,\cdot)\in
\mathcal{H}_{P^*}( M\times \mathbb{R}_-)$. Note that in the
selfadjoint case, we can extract a subsequence $\{t_{n_k}\}$ such that the limit function $a$ satisfies $a(\cdot,y,\cdot)\in \mathcal{H}_P( M\times \mathbb{R})$ \cite{PDavies}.
 }
 \end{rem}
 \begin{rem}\label{6Rem2} {\em
 Consider the complete (two-dimensional) Riemannian manifold $M$ that is constructed in \cite{PSt}. Then $M$ does not
admit nonconstant positive harmonic functions, $\lambda_0(-\Gd,M)=0$. Nevertheless, the heat operator does not admit any  $\lambda_ 0$-invariant positive solution. In particular, $M$ is stochastically incomplete (this construction disproves Stroock's conjecture concerning the existence of a $\lambda_ 0$-invariant positive solution).

On the other hand, since $M$ is Liouvillian, Remark~\ref{6Rem82} implies that Conjecture \ref{conjD} holds true on $M$.
Hence, the limit function
\begin{equation}\label{eq_axy}
a(x,y):=\lim_{t\to\infty} \frac{k_{-\Gd}^{M}(x,y,t)}{k_{-\Gd}^{M}(x_0,x_0,t)}
\end{equation}
 (which equals to the
constant function ${\bf 1}$) is not a $\lambda_ 0$-invariant
positive solution.  Compare this with \cite[Theorem~25]{D} and the
discussion therein above Lemma~26. In particular, it follows that the function
$$\gm(x):=\sup \left\{\sqrt{\frac{K_P^M(x, x,t)}{K_P^M(x_0, x_0,t)}}\;\;\Big|\; 1\leq t<\infty\right\}$$
is not slowly increasing in the sense of \cite{D}.
 }
 \end{rem}

Suppose that $P$ is symmetric (and $\gl_o=0$). Using \eqref{eq_ef_exp} and a standard exhaustion argument, it follows \cite{D} that for a fixed $x\in M$, the function $t\mapsto k_P^ M(x,x,t)$ is a nonincreasing log-convex function, and therefore, a
polarization argument implies  that the following strong ratio property holds true.
\begin{equation}\label{eq_xyt+s}
 \lim_{t\to\infty} \frac{k_P^ M(x,y,t+s)}{k_P^ M(x,y,t)}=1
\qquad  \forall x,y\in  M,\; s\in \mathbb{R}.
\end{equation}
In the nonsymmetric case, Corollary~\ref{asser_log} and the parabolic Harnack inequality imply:
\begin{lem}[\cite{PDavies}]\label{assliminfsup}
For every  $x,y\in  M$ and $s\in \mathbb{R}$, we have
that
\begin{equation}\label{eqskeleton11}
\liminf_{t\to\infty}
\frac{k_P^ M(x,y,t+s)}{k_P^ M(x,y,t)}\leq 1\leq
\limsup_{t\to\infty}
\frac{k_P^ M(x,y,t+s)}{k_P^ M(x,y,t)}\,.
\end{equation}
In particular, if $\,\lim_{t\to\infty}
[k_P^ M(x,y,t+s)/k_P^ M(x,y,t)]$ exists, it
equals to $1$.
\end{lem}
\begin{rem}\label{remellHar} {\em
If there exist $x_0,y_0\in  M$ and $0<s_0<1$ such that
\begin{equation}\label{eqconjDw17}
M(x_0,y_0,s_0):=\limsup_{t\to\infty}\frac{k_P^ M(x_0,y_0,t+s_0)}{k_P^ M(x_0,y_0,t)}<\infty,
\end{equation}
then by the parabolic Harnack inequality, for all $x,y,z,w\in
K\subset\subset  M$, $t>1$, we have the following Harnack
inequality of elliptic type:
\begin{multline}\label{eqstronghara}
k_P^ M(z,w,t)\leq
C_1k_P^ M(x_0,y_0,t + \frac{s_0}{2})  \leq\\
C_2k_P^ M(x_0,y_0,t - \frac{s_0}{2}) \leq
C_3k_P^ M(x,y,t).
\end{multline}
Similarly, (\ref{eqconjDw17}) implies that for all $x,y\in
 M$ and $r\in \mathbb{R}$:
\begin{align*}
0<m(x,y,r):=\liminf_{t\to\infty}
\frac{k_P^ M(x,y,t+r)}{k_P^ M(x_0,y_0,t)}
\leq \nonumber\\
\limsup_{t\to\infty}\frac{k_P^ M(x,y,t+r)}{k_P^ M(x_0,y_0,t)}
=M(x,y,r)<\infty. \end{align*}
Note that \eqref{eqconjDw17} is obviously satisfied in the symmetric case, and consequently  \eqref{eqstronghara} holds true \cite[Theorem~10]{Dheat}.
 }
 \end{rem}
It turns out that the validity of the strong limit property \eqref{eq_xyt+s} (in the nonsymmetric case) implies the validity of Davies' conjecture if in addition (\ref{eqoned}) is satisfied. We have:
\begin{lem}[\cite{PDavies}]\label{lem6}
 (a) The following assertions are equivalent:

\vskip 3mm

(i) For each $x,y\in M$ there exists a sequence $s_j\to
0$ of negative numbers such that for all $j\geq 1$, and $s=s_j$,
we have
\begin{equation}\label{eqconjDw27}
\lim_{t\to\infty}\frac{k_P^M(x,y,t+s)}{k_P^M(x,y,t)}=1.
\end{equation}

\vskip 3mm

(ii) The ratio limit in (\ref{eqconjDw27}) exists for any $x,y\in
M$ and $s\in \mathbb{R}$.

\vskip 3mm

(iii) Any limit function  $u(x,y,s)$ of the quotients
$\frac{k_P^M(x,y,t_n+s)}{k_P^M(x_0,x_0,t_n)}$
with $t_n \to  \infty$  does not depend on $s$ and has the form
$u(x,y)$, where  $u(\cdot,y) \in  \mathcal{C}_{P}(M)$
for every $y \in  M$ and $u(x, \cdot)  \in
\mathcal{C}_{P^*}(M)$ for every $x \in  M$.

\vskip 3mm

(b) If one assumes further that (\ref{eqoned}) is satisfied, then
Conjecture~\ref{conjD} holds true.

\vskip 3mm

Moreover, Conjecture~\ref{conjD} holds true if $M\subsetneqq \mathbb{R}^d$ is a smooth unbounded domain,  $P$ and $P^*$ are (up to the boundary) smooth operators, \eqref{eqconjDw27} holds true, and
\begin{equation}\label{eqoned0}
\dim \mathcal{C}^0_{P}( M)=\dim
\mathcal{C}^0_{P^*}( M)=1,
\end{equation}
where $\mathcal{C}^0_{P}( M)$ denotes the cone of all functions in
$\mathcal{C}_{P}( M)$ that vanish on $\partial  M$.

\end{lem}
 \begin{proof} (a) By Lemma~\ref{assliminfsup}, if the limit in
(\ref{eqconjDw27}) exists, then it is $1$.

\vskip 3mm

(i) $\Rightarrow$ (ii). Fix $x_0,y_0\in M$, and take
$s_0<0$ for which  the limit (\ref{eqconjDw27}) exists. It follows
that any limit solution $u(x,y,s)\in
\mathcal{H}_P(M\times \mathbb{R}_-)$ of a sequence
$\big\{\frac{k_P^M(x,y,t_n+s)}{k_P^M(x_0,y_0,t_n)}\big\}$
with $t_n\to \infty$  satisfies
$u(x_0,y_0,s + s_0) = u(x_0,y_0,s)$ for all $s < 0$. So,
$u(x_0,y_0,\cdot)$ is $s_0$-periodic. It follow from our
assumption and the continuity of $u$ that $u(x_0,y_0,\cdot)$ is
the constant function. Since this holds for all $x,y \in
M$ and $u$, it follows that (\ref{eqconjDw27}) holds for
any $x,y \in  M$ and $s \in  \mathbb{R}$.

\vskip 3mm

(ii) $\Rightarrow$ (iii). Fix $y\in M$. By Remark
\ref{6Rem21}, any limit function $u$ of the sequence
$\big\{\frac{k_P^M(x,y,t_n+s)}{k_P^M(x_0,x_0,t_n)}\big\}$
 with $t_n\to \infty$ belongs to
$\mathcal{H}_P(M\times \mathbb{R}_-)$. Since
\begin{equation}\label{eqfrac6}
\frac{k_P^M(x,y,t+s)}{k_P^M(x_0,x_0,t)}=
\frac{k_P^M(x,y,t)}{k_P^M(x_0,x_0,t)}
\frac{k_P^M(x,y,t+s)}{k_P^M(x,y,t)}\,,
\end{equation}
(\ref{eqconjDw27}) implies that such a $u$ does not depend on $s$.
Therefore, $u=u(x,y)$, where $u(\cdot,y)\in
\mathcal{C}_{P}(M)$ and $u(x, \cdot) \in
\mathcal{C}_{P^*}(M)$.

\vskip 3mm

(iii) $\Rightarrow$ (i). Write
\begin{equation}\label{eqquot67}
\frac{k_P^M(x,y,t+s)}{k_P^M(x,y,t)}=
\frac{k_P^M(x,y,t+s)}{k_P^M(x_0,x_0,t)}\,
\frac{k_P^M(x_0,x_0,t)}{k_P^M(x,y,t)}\,.
\end{equation}
Let $t_n\to \infty$ be a sequence such that the sequence
$\big\{\frac{k_P^M(x,y,t_n+s)}{k_P^M(x_0,x_0,t_n)}\big\}$
converges to a solution in $\mathcal{H}_P(M\times
\mathbb{R}_-)$. By our assumption, we have
$$\lim_{n\to\infty}\frac{k_P^M(x,y,t_n+s)}{k_P^M(x_0,x_0,t_n)}
=\lim_{n\to\infty}\frac{k_P^M(x,y,t_n)}{k_P^M(x_0,x_0,t_n)}=u(x,y)>0,$$
which together with (\ref{eqquot67}) implies (\ref{eqconjDw27})
for all $s\in \mathbb{R}$.

\vskip 3mm

\noindent (b) The uniqueness and (iii) imply that
$\frac{k_P^M(x,y,t+s)}{k_P^M(x_0,x_0,t)}\rightarrow
\frac{u(x)u^*(y)}{u(x_0)u^*(x_0)}$, where $u \in
\mathcal{C}_{P}(M)$ and $u^* \in
\mathcal{C}_{P^*}(M)$, and Conjecture~\ref{conjD} holds.
 \end{proof}
\begin{remark}\label{rem23}{\em
Assume that one of the assumptions of part (a) of Lemma~\ref{lem6} is satisfied. Then by the lemma,
any limit function of the sequence $\big\{\frac{k_P^M(x,y,t_n+s)}{k_P^M(x_0,x_0,t_n)}\big\}$ is of the form $a(x,y)$, where for every $y \in  M$, the function  $a(\cdot,y) \in  \mathcal{C}_{P}(M)$, and $a(x, \cdot)  \in
\mathcal{C}_{P^*}(M)$ for every $x \in  M$. But in general,
 $a(x,y)$ does not need to be a {\em
product} of solutions of the equations $Pu=0$ and $P^*u=0$, as is
demonstrated in \cite{CMS1}, in the hyperbolic space, and in
\cite[Example~4.2]{PDavies}.
 }
 \end{remark}

In the null-critical case we have:
\begin{lem}[\cite{PDavies}]\label{lem63}
Suppose that $P$ is null-critical, and for each $x,y\in
 M$ there exists a sequence  $\{s_j\}$  of negative
numbers such that $s_j \to 0$, and
\begin{equation}\label{eqconjDw71}
\liminf
_{t\to\infty}\frac{k_P^ M(x,y,t+s)}{k_P^ M(x,y,t)}\geq
1
\end{equation}
for $s=s_j$, $j=1,2,\ldots\,$. Then Conjecture~\ref{conjD} holds
true.
\end{lem}
 \begin{proof} Let $u(x,y,s)$ be a limit function of a sequence
$\big\{\frac{k_P^\mathcal{M}(x,y,t_n+s)}{k_P^\mathcal{M}(x_0,x_0,t_n)}\big\}$
with $t_n\to \infty$ and $s<0$.  By our assumption,
 $u(x,y,s+s_j)\geq u(x,y,s)$, and therefore,
$u_s(x,y,s)\leq 0$ for all $s<0$. Thus, $u(\cdot,y,s)$ (respect.,
$u(x,\cdot,s)$) is a positive supersolution of the equation $Pu=0$
(respect., $P^*u=0$) in $\mathcal{M}$. Since $P$ is critical, it
follows that $u(\cdot,y,s)\in \mathcal{C}_{P}(\mathcal{M})$
(respect., $u(x,\cdot,s)\in \mathcal{C}_{P^*}(\mathcal{M})$), and
hence $u_s(x,y,s)=0$. By the uniqueness, $u$ equals to
$\frac{\varphi(x)\varphi^*(y)}{\varphi(x_0)\varphi^*(x_0)}$, and
Conjecture~\ref{conjD} holds true.
 \end{proof}

 The large time behavior
of quotients of the heat kernel is obviously closely related to
the parabolic Martin boundary (for the parabolic Martin boundary
theory see \cite{CC,Doob,MJFA07}). The next result relates
Conjecture~\ref{conjD} and the parabolic Martin compactification
of $\mathcal{H}_P( M\times \mathbb{R}_-)$.
\begin{thm}[\cite{PDavies}]\label{thmconjD}
Assume that \eqref{eqconjDw17} holds true for some $x_0,y_0\in
 M$, and $s_0>0$. Then the following assertions are
equivalent:

(i) Conjecture~\ref{conjD} holds true for a fixed $x_0 \in
 M$.

\vskip 3mm

  (ii)
\begin{equation}\label{eqconjD2}
\lim_{t\to\infty}\frac{k_P^ M(x,y,t)}{k_P^ M(x_1,y_1,t)}
\end{equation}
exists, and the limit is positive for every $x,y,x_1,y_1\in
 M$.

\vskip 3mm

  (iii)
\begin{equation}\label{eqconjD1}
\lim_{t\to\infty}\frac{k_P^ M(x,y,t)}{k_P^ M(y,y,t)}
\, , \quad \mbox{and} \quad
\lim_{t\to\infty}\frac{k_P^ M(x,y,t)}{k_P^ M(x,x,t)}
\end{equation}
exist, and these ratio limits are positive for every $x,y\in
 M$.

\vskip 3mm

(iv) For any $y\in  M$ there is a unique nonzero
parabolic Martin boundary point $\bar y$ for the equation  $Lu=0$
in $ M\times \mathbb{R}$ which corresponds to any
sequence of the form $\{(y,-t_n)\}_{n=1}^\infty$ such that $t_n\to
\infty$, and for any $x\in  M$ there is a unique nonzero
parabolic Martin boundary point $\bar x$ for the equation
$u_t+P^*u=0$ in $ M\times \mathbb{R}$ which corresponds
to any sequence of the form $\{(x,-t_n)\}_{n=1}^\infty$ such that
$t_n\to \infty$.

\vskip 3mm

Moreover, if Conjecture~\ref{conjD} holds true, then for any fixed
$y\in  M$ (respect.,  $x \in  M$), the limit
function $a(\cdot,y)$ (respect., $a(x,\cdot)$) is a positive
solution of the equation $Pu = 0$ (respect., $P^*u = 0$).
Furthermore, the Martin functions of part (iv) are time
independent, and (\ref{eqconjDw27}) holds for all $x,y \in
 M$ and $s \in \mathbb{R}$.
\end{thm}
  \begin{proof}  (i) $\Rightarrow$ (ii) follows from the identity
$$\frac{k_P^ M(x,y,t)}{k_P^ M(x_1,y_1,t)}=
\frac{k_P^ M(x,y,t)}{k_P^ M(x_0,x_0,t)}\cdot
\left(\frac{k_P^ M(x_1,y_1,t)}{k_P^ M(x_0,x_0,t)}\right)^{-1}.$$

\vskip 3mm

(ii) $\Rightarrow$ (iii). Take $x_1=y_1=y$ and $x_1=y_1=x$,
respectively.

\vskip 3mm

(iii) $\Rightarrow$ (iv). It is well known that the Martin
compactification does not depend on the fixed reference point
$x_0$. So, fix $y\in  M$ and take it also as a reference
point. Let $\{-t_n\}$ be a sequence such that $t_n\to \infty$ and
such that the Martin sequence
$\big\{\frac{k_P^ M(x,y,t+t_n)}{k_P^ M(y,y,t_n)}\big\}$
converges to a Martin function $K_P^ M(x,\bar{y},t)$. By
our assumption, for any $t$ we have
$$\lim_{n\to\infty}\frac{k_P^ M(x,y,t+t_n)}{k_P^ M(y,y,t+t_n)}=
\lim_{\tau\to\infty}\frac{k_P^ M(x,y,\tau)}{k_P^ M(y,y,\tau)}=
b(x)>0,$$ where $b$ does not depend on the sequence $\{-t_n\}$. On
the other hand,
$$\lim_{n\to\infty}\frac{k_P^ M(y,y,t+t_n)}{k_P^ M(y,y,t_n)}=
K_P^ M(y,\bar{y},t)=f(t).$$
 Since
$$\frac{k_P^ M(x,y,t+t_n)}{k_P^ M(y,y,t_n)}=
\frac{k_P^ M(x,y,t+t_n)}{k_P^ M(y,y,t+t_n)}\cdot
\frac{k_P^ M(y,y,t+t_n)}{k_P^ M(y,y,t_n)},$$ we
have
$$K_P^ M(x,\bar{y},t)=b(x)f(t).$$
 By separation of variables, there exists a constant
$\lambda$ such that
$$Pb-\lambda b=0 \quad \mbox{ on }  M,
\qquad f'+\lambda f=0 \quad \mbox{ on } \mathbb{R},\;\; f(0)=1.
$$
 Since $b$ does not depend on the sequence
$\{-t_n\}$, it follows in particular,  that $\lambda$ does not
depend on this sequence. Thus, $ \lim_{\tau\to \infty}
\frac{k_P^ M(y,y,t+\tau)}{k_P^ M(y,y,\tau)}=f(t)=\mathrm{e}^{-\lambda
t}$. Lemma~\ref{assliminfsup} implies that $\lambda=0$.  It
follows that $b$ is a positive solution of the equation $Pu=0$,
and
\begin{equation}\label{6eq1}
K_P^ M(x,\bar{y},t)=\lim_{\tau\to -
\infty}\frac{k_P^ M(x,y,t-\tau)}{k_P^ M(y,y,-\tau)}=
b(x).
\end{equation}
The dual assertion can be proved similarly.

\vskip 3mm

(iv) $\Rightarrow$ (i). Let $K_P^ M(x,\bar{y},t)$ be the
Martin function given in (iv), and $s_0>0$ such that
$K_P^ M(x_0,\bar{y},s_0/2)>0$. Consequently,
$K_P^ M(x,\bar{y},s)>0$  for $s\geq s_0$. Using the
substitution $\tau=s+s_0$ we obtain \begin{multline}      \lim_{\tau\to
\infty}
\frac{k_P^ M(x,y,\tau)}{k_P^ M(x_0,x_0,\tau)} =
\lim_{s\to \infty}  \left \{
\frac{k_P^ M(x,y,s+s_0)}{k_P^ M(y,y,s)}\right.\times\\[3mm]
\left.\frac{k_P^ M(y,y,s)}{k_P^ M(x_0,y,s + 2s_0)}
 \frac{k_P^ M(x_0,y,s + 2s_0)}{k_P^ M(x_0,x_0,s + s_0)}  \right
\} =
  \frac{K_P^ M(x,\bar{y},s_0)K_{P^*}^ M(\overline{x_0},y,s_0)}
{K_P^ M(x_0,\bar{y},2s_0)}.
 \end{multline}
The last assertion of the theorem follows from  (\ref{6eq1}) and
Lemma \ref{lem6}.
 \end{proof}
 Finally we mention a problem which
was raised by Burdzy and Salisbury \cite{BS} for $P=-\Delta$ and
$\mathcal{M}\subset \mathbb{R}^d$.
\begin{question}[\cite{PDavies}]\label{questBS}
Assume that $\lambda_0=0$. Determine which minimal functions (in the sense of Martin's boundary) in $\mathcal{C}_{P}(\mathcal{M})$ are minimal in
$\mathcal{H}_P(\mathcal{M}\times \mathbb{R_-})$. In particular, is it true that in the critical case, the ground state $\vgf$ is minimal in $\mathcal{H}_P(\mathcal{M}\times \mathbb{R_-})$?
\end{question}

\mysection{Comparing decay of critical and
subcritical heat kernels}\label{sect1DMY}
In this section we are concerned with the large time behavior of the
heat kernel~$k_P^M$ with regards to the criticality versus
subcriticality property of the operator~$P$. Since for any fixed
$x,y\in M$, $x \not = y$, we have that $k_{P}^M(x,y,\cdot)\in L^1(\mathbb{R}_+)$
if and only if $P$ is subcritical, it is natural to conjecture
that \emph{under some assumptions} the heat kernel of a subcritical operator~$P_+$~in $M$ decays
(in time) faster than the heat kernel of a critical
operator~$P_0$~in $M$. Hence, our aim is to discuss the following conjecture in {\em general} settings.
\begin{conjecture}[\cite{fkp}]\label{conjMain}
Let $P_+$ and $P_0$ be respectively subcritical and critical
operators in $M$. Then
\begin{equation}\label{eqconjMain}
\lim_{t\to\infty}\frac{k_{P_+}^M(x,y,t)}{k_{P_0}^M(x,y,t)}=0
\end{equation}
locally uniformly in $M\times M$.
\end{conjecture}
Conjecture~\ref{conjMain} was stimulated by the following conjecture of D.~Kr\-ej\v{c}\-i\v{r}\'{\i}k and E.~Zuazua \cite{KZ}:

{\em  Let $P_+$ and $P_0$ be, respectively, {\em selfadjoint} subcritical and critical operators defined on $L^2(M,\mathrm{d}x)$. Then
\begin{equation}\label{KZqconj}
\lim_{t\to\infty}\frac{\|\mathrm{e}^{-P_+t}\|_{L^2(M,W\,\mathrm{d}x)
\to L^2(M,\mathrm{d}x)}}
{\|\mathrm{e}^{-P_0t}\|_{L^2(M,W\,\mathrm{d}x)\to L^2(M,\mathrm{d}x)}} =0
\end{equation}
for some positive weight function $W$}.

\vskip 3mm
\begin{remark}\label{rem_89}{\em
Theorem~\ref{thm1} implies that Conjecture~\ref{conjMain}
obviously holds true if~$P_0$ is positive-critical. Moreover, by Corollary~\ref{asser_log} the conjecture also holds true if $\gl_0(P_+,M)>0$. Therefore, throughout this section we assume that $\gl_0(P_+,M)=0$, and $P_0$ is null-critical in $M$.
 }
 \end{remark}

\begin{example}\label{ex_M}{\em
In \cite{M_large_time} M.~Murata
obtained the exact asymptotic for the heat kernels of
nonnegative Schr\"odinger operators with {\em short-range} (real)
potentials defined on~$\mathbb{R}^d$, $d\geq 1$.
In particular, \cite[theorems~4.2 and 4.4]{M_large_time} imply that Conjecture~\ref{conjMain}
holds true for such operators.
 }
\end{example}
The following theorem deals with the {\em symmetric} case.
\begin{theorem}[\cite{fkp}]\label{mainthmFKP}
Let the subcritical operator~$P_+$
and the critical operator~$P_0$ be symmetric in $M$ of the form
\begin{equation}\label{PSthm}
Pu=-m^{-1} \mathrm{div}(mA\nabla u) +Vu .
\end{equation}

Assume that $A_+$ and $A_0$, the sections on $M$ of $\mathrm{End}(TM)$, and the weights $m_+$ and $m_0$, corresponding to  $P_+$ and $P_0$, respectively, satisfy the following matrix inequality
\begin{equation}\label{A+leqA0}
m_+(x)A_+(x)\leq C m_0(x)A_0(x) \qquad  \mbox{for a.e. } x\in M,
\end{equation}
where $C$ is a positive constant. Assume further that for some fixed $y_1\in M$ there exists a positive constant $C$ satisfying the following condition: for each $x\in M$ there
exists $T(x)>0$ such that
\begin{equation}\label{Ass1m}
    k_{P_+}^M(x,y_1,t)\leq C k_{P_0}^M(x,y_1,t)\qquad \forall t>T(x).
\end{equation}
Then
\begin{equation}\label{eqconjMain1}
\lim_{t\to\infty}\frac{k_{P_+}^M(x,y,t)}{k_{P_0}^M(x,y,t)}=0
\end{equation}
locally uniformly in $M\times M$.
 \end{theorem}
\begin{proof}
Recall that in light of Remark~\ref{rem_89} we assume that $\lambda_0(P_+,M)= 0$. Suppose to the contrary that for some $x_0,y_0\in M$ there exists a sequence $\{t_n\}$ such that $t_n \to \infty$, and
    \begin{equation}\label{eqconjMain2}
\lim_{n\to\infty}\frac{k_{P_+}^M(x_0,y_0,t_n)}{k_{P_0}^M(x_0,y_0,t_n)}=K>0.
 \end{equation}
Consider the sequence of functions $\{u_n\}_{n=1}^\infty$ defined by
$$ u_n(x,s):=\frac{k_{P_+}^M(x,y_0,t_n+s)}{k_{P_0}^M(x_0,y_0,t_n)}\qquad x\in M,\; s\in \mathbb{R}.$$
We note that
$$ u_n(x,s)=\frac{k_{P_+}^M(x,y_0,t_n+s)}{k_{P_+}^M(x_0,y_0,t_n)}\times \frac{k_{P_+}^M(x_0,y_0,t_n)}{k_{P_0}^M(x_0,y_0,t_n)}\;.$$
Therefore, by assumption \eqref{eqconjMain2} and Remark~\ref{6Rem21} it follows that we may subtract a subsequence which we rename by $\{u_n\}$ such that
$$\lim_{n\to\infty} u_n(x,s)=u_+(x,s),$$
where $u_+\in \mathcal{H}_{P_+}(M\times \mathbb{R})$ and $u_+\gneqq 0$.

On the other hand,
\begin{align*}
v_n(x):=\frac{k_{P_+}^M(x,y_0,t_n)}{k_{P_0}^M(x_0,y_0,t_n)}
=u_n(x,s)\frac{k_{P_+}^M(x,y_0,t_n)}{k_{P_+}^M(x,y_0,t_n+s)}\;.
\end{align*}
By our assumption, $\lambda_0(P_+,M)= 0$ and $P_+$ is symmetric, therefore \eqref{eq_xyt+s} implies that
$$\lim_{n\to\infty}\frac{k_{P_+}^M(x,y_0,t_n)}{k_{P_+}^M(x,y_0,t_n+s)}=1.$$
Therefore,
$$\lim_{n\to\infty}v_n(x)=\lim_{n\to\infty}u_n(x,s)=u_+(x,s),$$
and $u_+$ does not depend on $s$, and hence $u_+$ is a positive solution of the elliptic
equation $P_+u=0$ in $M$ and we have
\begin{equation}\label{eq0}
    \lim_{n\to\infty}\frac{k_{P_+}^M(x,y_0,t_n)}{k_{P_0}^M(x_0,y_0,t_n)}=u_+(x).
\end{equation}
On the other hand, by Remark~\ref{6Rem82} we have
\begin{equation}\label{eq1}
\lim_{n\to\infty}\frac{k_{P_0}^M(x,y_0,t_n)}{k_{P_0}^M(x_0,y_0,t_n)}
= \frac{\varphi(x)}{\varphi(x_0)} =: u_0(x),
\end{equation}
where $\varphi$ is the ground state of $P_0$.

Combining \eqref{eq0} and \eqref{eq1}, we obtain
 \begin{align}\label{eq2}
\lim_{n\to\infty}  \frac{k_{P_+}^M(x,y_0,t_n)}{k_{P_0}^M(x,y_0,t_n)}= \lim_{n\to\infty} \left\{
 \dfrac{\frac{k_{P_+}^M(x,y_0,t_n)}{k_{P_0}^M(x_0,y_0,t_n)}}
 {\frac{k_{P_0}^M(x,y_0,t_n)}{k_{P_0}^M(x_0,y_0,t_n)}}\right\}=\frac{u_+(x)}{u_0(x)}.
 \end{align}
 On the other hand, by assumption \eqref{Ass1m} and the parabolic Harnack inequality
there exists a positive constant~$C_1$
which depends on $P_+$, $P_0$, $y_0$, and $y_1$
such that
\begin{multline}\label{Ass2}
    C_1^{-1}k_{P_+}^M(x,y_0,t-1)\leq k_{P_+}^M(x,y_1,t)\\\leq C k_{P_0}^M(x,y_1,t) \leq
CC_1 k_{P_0}^M(x,y_0,t+1)    \quad \forall x\in M, t>T(x).
\end{multline}
Moreover, by \eqref{eq_xyt+s} we have
\begin{equation}\label{eqconjDw91}
\lim_{t\to\infty}\frac{k_{P_+}^M(x,y_0,t-1)}{k_{P_+}^M(x,y_0,t)}=1, \; \mbox{ and }\; \lim_{t\to\infty}\frac{k_{P_0}^M(x,y_0,t+1)}{k_{P_0}^M(x,y_0,t)}=1
\quad \forall  x\in M.
\end{equation}
Therefore, \eqref{Ass2} and \eqref{eqconjDw91} imply that there exists $C_0>0$ such that
\begin{equation}\label{Ass3}
    k_{P_+}^M(x,y_0,t)\leq C_0 k_{P_0}^M(x,y_0,t)    \qquad \forall x\in M, t>T(x).
\end{equation}
 Consequently, \eqref{eq2} and \eqref{Ass3} imply that
 $$u_+(x)\leq C_0 u_0(x) = \tilde{C}_0 \varphi(x)\qquad \forall x\in M. $$
 Therefore, using \eqref{A+leqA0} we obtain
 \begin{equation}\label{u+leu_0}
(u_+)^2(x) m_+(x)A_+(x)\leq C_2\varphi^2(x) m_0(x)A_0(x)\qquad  \mbox{for a.e. } x\in M,
\end{equation}
where $C_2>0$ is a positive constant. Thus, Theorem~\ref{mainthmLt} implies that $P_+$ is critical in $M$ which is a contradiction.
The last statement of the theorem follows from the parabolic
Harnack inequality and parabolic regularity.
  \end{proof}
By the generalized maximum principle, assumption \eqref{Ass1m} in Theorem~\ref{mainthmFKP} is satisfied with $C=1$ if $P_+=P_0+V$, where  $P_0$ is a critical operator on $M$ and $V$ is any {\em nonnegative} potential. Note that if the potential is in addition nontrivial,
then $P_+$ is indeed subcritical in $M$. Therefore, we have
\begin{corollary}[\cite{fkp}]\label{nonnegpert}
Let $P_0$ be a symmetric operator of the form \eqref{PSthm} which is critical in $M$,
and let $P_+:=P_0+V_1$, where  $V_1$ is a nonzero nonnegative potential.  Then
\begin{equation}\label{eqconjMain5}
\lim_{t\to\infty}\frac{k_{P_+}^M(x,y,t)}{k_{P_0}^M(x,y,t)}=0
\end{equation}
locally uniformly in $M\times M$.
\end{corollary}
Next, we discuss the nonsymmetric case. We study two cases where Davies' conjecture implies Conjecture~\ref{conjMain}. First, we show that in the nonsymmetric case, the result of Corollary~\ref{nonnegpert} for positive perturbations of a critical operator~$P_0$ still holds provided that the validity of Davies' conjecture (Conjecture~\ref{conjD})
is assumed instead of the symmetry hypothesis. More precisely, we have
\begin{theorem}[\cite{fkp}]\label{thm_nonselfadj}
Let $P_0$ be a critical operator in $M$, and let $P_+=P_0+V$, where~$V$ is any nonzero nonnegative potential on $M$. Assume that
Davies' conjecture (Conjecture~\ref{conjD}) holds true for both $k_{P_0}^M$ and $k_{P_+}^M$. Then
\begin{equation}\label{eqconjMain6}
\lim_{t\to\infty}\frac{k_{P_+}^M(x,y,t)}{k_{P_0}^M(x,y,t)}=0
\end{equation}
locally uniformly in $M\times M$.
 \end{theorem}
\begin{proof}
 Recall that in light of Remark~\ref{rem_89} we assume that $\lambda_0(P_+,M)= 0$. Suppose to the contrary that for some $x_0,y_0\in M$ there exists a sequence $\{t_n\}$ such that $t_n \to \infty$ and
    \begin{equation}\label{eqconjMain2n}
\lim_{n\to\infty}\frac{k_{P_+}^M(x_0,y_0,t_n)}{k_{P_0}^M(x_0,y_0,t_n)}=K>0.
 \end{equation}
Consider the functions $v_+$ and  $v_0$ defined by
$$ v_+(x,t):=\frac{k_{P_+}^M(x,y_0,t)}{k_{P_+}^M(x_0,y_0,t)}\,,\quad v_0(x,t):=\frac{k_{P_0}^M(x,y_0,t)}{k_{P_0}^M(x_0,y_0,t)} \qquad x\in M, t>0.$$
By our assumption,
$$\lim_{t\to\infty} v_+(x,t)=u_+(x),\qquad \lim_{t\to\infty}
v_0(x,t)=u_0(x),$$
where $u_+\in \mathcal{C}_{P_+}(M)$ and $u_0\in \mathcal{C}_{P_0}(M)$.

On the other hand, by the generalized maximum principle
\begin{equation}\label{gmp}
\frac{k_{P_+}^M(x,y_0,t)}{k_{P_0}^M(x,y_0,t)}\leq 1.
\end{equation}
Therefore,
\begin{equation}\label{eqm1}
\frac{k_{P_+}^M(x_0,y_0,t_n)}{k_{P_0}^M(x_0,y_0,t_n)}\times
\dfrac{\frac{k_{P_+}^M(x,y_0,t_n)}{k_{P_+}^M(x_0,y_0,t_n)}}
{\frac{k_{P_0}^M(x,y_0,t_n)}{k_{P_0}^M(x_0,y_0,t_n)}} =
  \frac{k_{P_+}^M(x,y_0,t_n)}{k_{P_0}^M(x,y_0,t_n)}\leq 1.
\end{equation}
Letting $n\to \infty$ we obtain
\begin{equation}\label{eqm2}
Ku_+(x) \leq u_0(x) \qquad x\in M.
\end{equation}
It follows that $v(x):=u_0(x)-Ku_+(x)$ is a nonnegative supersolution of the  equation $P_0 u=0$ in $M$ which is not a solution. In particular, $v\neq 0$. By the strong maximum principle $v(x)$ is a strictly positive supersolution of the  equation $P_0 u=0$ in $M$ which is not a solution. This contradicts the criticality of $P_0$ in $M$.
\end{proof}
The second nonsymmetric result concerns semismall perturbations.
\begin{theorem}[\cite{fkp}]\label{thm_ssp}
Let $P_+$ and $P_0=P_++V$ be a subcritical operator and a critical operator in $M$, respectively. Suppose that $V$ is a semismall perturbation of the operator $P_+^*$ in
$M$. Assume further that Davies' conjecture (Conjecture~\ref{conjD}) holds
true for both $k_{P_0}^M$ and $k_{P_+}^M$ and that \eqref{Ass1m}
holds true. Then
\begin{equation}\label{eqconjMain6a}
\lim_{t\to\infty}\frac{k_{P_+}^M(x,y,t)}{k_{P_0}^M(x,y,t)}=0
\end{equation}
locally uniformly in $M\times M$.
 \end{theorem}
\begin{proof}
 Recall that we (may) assume that $\lambda_0(P_+,M)= 0$.
Assume to the contrary that for some $x_0,y_0\in M$ there exists a
sequence $\{t_n\}$ such that $t_n \to \infty$ and
    \begin{equation}\label{eqconjMain2na}
\lim_{n\to\infty}\frac{k_{P_+}^M(x_0,y_0,t_n)}{k_{P_0}^M(x_0,y_0,t_n)}=K>0.
 \end{equation}
Consider the functions $v_+$ and  $v_0$ defined by
\begin{equation}\label{eqconjMain2nab}
v_+(x,t):=\frac{k_{P_+}^M(x,y_0,t)}{k_{P_+}^M(x_0,y_0,t)}\,,\quad v_0(x,t):=\frac{k_{P_0}^M(x,y_0,t)}{k_{P_0}^M(x_0,y_0,t)} \qquad x\in M, t>0.
\end{equation}
By our assumption,
$$\lim_{t\to\infty} v_+(x,t)=u_+(x),\qquad \lim_{t\to\infty} v_0(x,t)=u_0(x),$$
where $u_+\in \mathcal{C}_{P_+}(M)$ and $u_0\in
\mathcal{C}_{P_0}(M)$.

On the other hand, by assumption~\eqref{Ass1m} we have for $t>T(x)$
\begin{equation}\label{gmpssp}
\frac{k_{P_+}^M(x,y_0,t)}{k_{P_0}^M(x,y_0,t)} = \frac{k_{P_+}^M(x,y_1,t)}{k_{P_0}^M(x,y_1,t)}
\times \dfrac{\frac{k_{P_+}^M(x,y_0,t)}{k_{P_+}^M(x,y_1,t)}}{\frac{k_{P_0}^M(x,y_0,t)}{k_{P_0}^M(x,y_1,t)}} \leq C \frac{k_{P_+}^M(x,y_0,t)}{k_{P_+}^M(x,y_1,t)}\times\frac{k_{P_0}^M(x,y_1,t)}{k_{P_0}^M(x,y_0,t)}
\,.
\end{equation}
By our assumption on Davies' conjecture, we have for a fixed $x$
\begin{equation}\label{gmpssp6}
\lim_{t\to\infty}\frac{k_{P_+}^M(x,y_0,t)}{k_{P_+}^M(x,y_1,t)}= \frac{u_+^*(y_0)}{u_+^*(y_1)}\;,\qquad
\lim_{t\to\infty}\frac{k_{P_0}^M(x,y_1,t)}{k_{P_0}^M(x,y_0,t)}=\frac{u_0^*(y_1)}{u_0^*(y_0)},
\end{equation}
where $u_+^*$ and $u_0^*$ are positive solutions of the equation $P_+^*u=0$ and  $P_0^*u=0$ in $M$, respectively. By the elliptic Harnack inequality there exists a positive constant $C_1$ (depending on $P_+^*,P_0^*, y_0,y_1$ but not on $x$) such that
\begin{equation}\label{gmpssp61}
 \frac{u_+^*(y_0)}{u_+^*(y_1)}\leq C_1,\qquad
\frac{u_0^*(y_1)}{u_0^*(y_0)}\leq C_1.
\end{equation}
Therefore,  \eqref{gmpssp} and \eqref{gmpssp61} imply that
\begin{equation}\label{eqm1ssp9}
  \frac{k_{P_+}^M(x,y_0,t_n)}{k_{P_0}^M(x,y_0,t_n)}\leq 2CC_1^2
\end{equation}
for $n$ sufficiently large (which might depend on $x$).

Therefore,
\begin{equation}\label{eqm1ssp}
\frac{k_{P_+}^M(x_0,y_0,t_n)}{k_{P_0}^M(x_0,y_0,t_n)}\times \dfrac{\frac{k_{P_+}^M(x,y_0,t_n)}{k_{P_+}^M(x_0,y_0,t_n)}}
{\frac{k_{P_0}^M(x,y_0,t_n)}{k_{P_0}^M(x_0,y_0,t_n)}}
=
  \frac{k_{P_+}^M(x,y_0,t_n)}{k_{P_0}^M(x,y_0,t_n)}\leq 2CC_1^2
  \,.
\end{equation}
Letting $n\to \infty$ and using~\eqref{eqconjMain2na} and~\eqref{eqconjMain2nab},
we obtain
\begin{equation}\label{eqm2ssp}
Ku_+(x) \leq 2CC_1^2 u_0(x) \qquad x\in M
\,.
\end{equation}
On the other hand, since $V$ is a semismall perturbation of
$P_+^*$ in $M$, Theorem~\ref{thmssp} implies that $u_0(x)\asymp
G_{P_+}^M(x,y_0)$ in $M\setminus \overline{B(y_0,\delta)}$,
with some positive~$\delta$.
Consequently,
\begin{equation}\label{eqm3ssp}
u_+(x) \leq  C_2G_{P_+}^M(x,y_0)\qquad x\in M \setminus \overline{B(y_0,\delta)}
\end{equation}
for some $C_2>0$. In other words, $u_+$ is a global positive solution of the equation $P_+u=0$ in $M$ which has
minimal growth in a neighborhood of infinity in $M$. Therefore
$u_+$ is a ground state of the equation $P_+u=0$ in $M$, but this
contradicts the subcriticality of $P_+$ in $M$.
\end{proof}
%

\mysection{On the equivalence of heat kernels}\label{secequiv}
In this section we study a general question concerning the equivalence of heat kernels which in turn will give sufficient conditions for the validity of the boundedness condition \eqref{Ass1m} that is assumed in theorems~\ref{mainthmFKP} and \ref{thm_ssp}.

\begin{definition}\label{defequiv}{\em
Let $P_{i},\,i=1,2$, be two elliptic operators  of the form \eqref{P} that are defined on
$M$ and satisfy $\lambda_0(P_i,M) =0$ for $i=1,2 $. We say that the heat kernels
$k_{P_1}^M(x,y,t)$ and $k_{P_2}^M(x,y,t)$ are
{\em equivalent} (respectively, {\em semiequivalent}) if $$k_{P_1}^M \asymp
k_{P_2}^M \qquad   \mbox{ on } M \times M \times (0,\infty)$$
(respectively, $$ k_{P_1}^M (\cdot,y_0,\cdot)\asymp k_{P_2}^M(\cdot,y_0,\cdot) \quad \mbox{ on }
\;M\times (0,\infty)  $$
for some fixed $y_0\in M$).
 }
\end{definition}

There is an intensive literature dealing with (almost optimal) conditions under which two positive (minimal) Green functions are equivalent or semi\-equivalent
(see \cite{a97,Msemismall,P89,P99} and the references therein). On the other hand, sufficient conditions for the equivalence of heat kernels are known only in a few cases (see \cite{LS,MS,Zhang}). In particular, it seems that the answer to the following conjecture is not known.
\begin{conjecture}[\cite{fkp}]\label{conjequival}
Let $P_1$ and $P_2$ be two subcritical operators of the form \eqref{P} that are
defined on a Riemannian manifold $M$ such that $P_1=P_2$ outside a compact set in $M$, and $\lambda_0(P_i,M) =0$ for $i=1,2 $ .  Then
$k_{P_1}^M$ and $k_{P_2}^M$ are equivalent.
\end{conjecture}
\begin{remark}\label{rem_green_lambda}{\em
Suppose that $P_1$ and $P_2$ satisfy the assumption of Conjecture~\ref{conjequival}. B.~Devyver and the author proved recently that there exists $C>0$ such that \begin{equation}\label{eq_Gr_eq}
    C^{-1}G_{P_2-\gl}^M(x,y) \leq G_{P_1-\gl}^M(x,y) \leq CG_{P_2-\gl}^M(x,y)\;\; \forall x,y\in M, \mbox{ and } \gl\leq 0.
\end{equation}
Clearly, by \eqref{def.Gr}, the above estimate \eqref{eq_Gr_eq} is a necessary condition for the validity of Conjecture~\ref{conjequival}.
 }
 \end{remark}
It is well known that certain $3\mbox-G$ inequalities imply the
equivalence of Green functions, and the notions of small and
semismall perturbations is based on this fact.
Moreover, small
(respectively, semismall) perturbations are sufficient conditions and
in some sense also necessary conditions for the equivalence
(respectively, semiequivalence) of the Green functions
\cite{Msemismall,P89,P99}. Therefore, it is natural to introduce an analog definition for heat kernels (cf. ~\cite{Zhang}).
\begin{definition}\label{def3kineq}{\em
Let $P$ be a subcritical operator in $M$.
We say that a potential $V$ is a {\em $k$-bounded  perturbation}
(respectively, {\em $k$-semibounded perturbation})
with respect to the heat kernel $k_{P}^M(x,y,t)$ if there exists a positive constant $C$ such that the following $3\mbox-k$ inequality is satisfied:
\begin{equation}\label{eq3kineq}
\int_0^t\int_M k_{P}^M(x,z,t-s)|V(z)|k_{P}^M(z,y,s)\dz\,\mathrm{d}s \leq C k_{P}^M(x,y,t)
\end{equation}
for all $x,y\in M$ (respectively,  for a fixed $y\in M$, and all $x\in M$) and $t>0$.
 }
\end{definition}
The following result shows that the notion of $k$-(semi)boundedness
is naturally related to the (semi)equivalence of heat kernels.
\begin{theorem}[\cite{fkp}]\label{thmbounded}
Let $P$ be a subcritical operator in $M$,
and assume that the potential $V$ is $k$-bounded  perturbation
(respectively, $k$-semibounded perturbation)
with respect to the heat kernel $k_{P}^M(x,y,t)$.
Then there exists $c>0$ such that for any $|\varepsilon|<c$ the heat kernels
$k_{P+\varepsilon V}^M(x,y,t)$ and $k_{P}^M(x,y,t)$ are
equivalent (respectively, semiequivalent).
\end{theorem}
\begin{proof}
Consider the iterated kernels
$$
k_{P}^{(j)}(x,y,t):= \left\{
  \begin{array}{ll}
    k_{P}^M(x,y,t) & j=0, \\[4mm]
    \int_0^t  \int_M  k_P^{(j-1)}(x,z,t-s)  V(z)
  k_{P}^M(z,y,s)
  \dz\,\mathrm{d}s &  j\geq 1.
  \end{array}
\right.
$$
Using~\eqref{eq3kineq} and an induction argument, it follows that
\begin{multline*}
  \sum_{j=0}^\infty |\varepsilon|^j|k_{P}^{(j)}(x,y,t)|\\
  \leq
  \left(1+C|\varepsilon|+C^2|\varepsilon|^2+\dots\right) k_{P}^M(x,y,t)
  = \frac{1}{1-C|\varepsilon|} \, k_{P}^M(x,y,t)
\end{multline*}
provided that $|\varepsilon|<C^{-1}$. Consequently, for such $\varepsilon$ the Neumann series
$$\sum_{j=0}^\infty (-\varepsilon)^j k_{P}^{(j)}(x,y,t)$$ converges locally uniformly in $M\times M \times \mathbb{R}_+$ to $k_{P+\varepsilon V}^M(x,y,t)$ which in turn implies that Duhamel's formula
\begin{multline}\label{Duhamel}
  k_{P+\varepsilon V}^M(x,y,t) = k_{P}^M(x,y,t)
  - \\
  \varepsilon    \int_0^t  \int_M   k_{P}^M(x,z,t-s)  V(z)
  k_{P+\varepsilon V}^M(z,y,s)
  \dz\,\mathrm{d}s
\end{multline}
is valid. Moreover, we have
$$k_{P+\varepsilon V}^M(x,y,t) \leq \frac{1}{1-C|\varepsilon|} \, k_{P}^M(x,y,t).$$
The lower bound
$$C_1k_{P}^M(x,y,t) \leq k_{P+\varepsilon V}^M(x,y,t)$$
(for $|\varepsilon|$ small enough) follows from the upper bound using \eqref{Duhamel} and \eqref{eq3kineq}.
\end{proof}
\begin{corollary}[\cite{fkp}]\label{cor7}
Assume that $P$ and $V$ satisfy the conditions of Theorem~\ref{thmbounded}, and suppose
further that $V$ is nonnegative. Then there exists $c>0$ such that for any $\varepsilon>-c $ the heat kernels
$k_{P+\varepsilon V}^M(x,y,t)$ and $k_{P}^M(x,y,t)$ are
equivalent (respectively, semiequivalent).
\end{corollary}
\begin{proof}
By Theorem~\ref{thmbounded} the heat kernels
$k_{P+\varepsilon V}^M(x,y,t)$ and $k_{P}^M(x,y,t)$
are equivalent (respectively, semiequivalent) for any
$|\varepsilon|<c$.

Recall that by the generalized maximum principle,
\begin{equation}\label{eq_gmp}
k_{P+\varepsilon V}^M(x,y,t)\leq k_{P}^M(x,y,t)\qquad \forall \varepsilon>0.
\end{equation}
On the other hand, given a potential $W$ (not necessarily of definite sign), and $0\leq \ga\leq 1$, denote $P_\ga:= P+ \ga W$, and assume that $P_j\geq 0$ in $M$ for $j=0,1$. Then  for $0\leq \ga\leq 1$ we have $P_\ga\geq 0$ in $M$, and the following inequality holds \cite{P90}
\begin{equation}\label{eqkconvex}
k_{P_\ga}^M(x,y,t)\leq [k_{P_0}^M(x,y,t)]^{1-\ga}[k_{P_1}^M(x,y,t)]^\ga \qquad \forall\, x,y\in M,\; t>0.
\end{equation}
Using \eqref{eq_gmp} and \eqref{eqkconvex} we obtain the desired equivalence $k_{P+\varepsilon V}^M \asymp k_{P}^M$ also for $\varepsilon\geq c$.
\end{proof}
Finally we have:
\begin{theorem}[\cite{fkp}]\label{thmcond1}
Let $P_0$ be a critical operator in $M$. Assume that $V=V_+ - V_-$ is a potential such that $V_\pm \geq 0$ and  $P_+:=P_0+V$ is subcritical in $M$.

Assume further that $V_-$ is $k$-semibounded perturbation with respect to the heat kernel $k_{P_+}^M(x,y_1,t)$. Then condition \eqref{Ass1m} is satisfied uniformly in $x$. That is, there exist positive constants $C$ and $T$ such that
\begin{equation}\label{Ass1n}
    k_{P_+}^M(x,y_1,t)\leq C k_{P_0}^M(x,y_1,t)\qquad \forall x\in M, t>T.
\end{equation}
  \end{theorem}
\begin{proof}
 It follows from  Corollary~\ref{cor7}, that the heat kernels $k_{P_+}^M(x,y_1,t)$ and $k_{P_+ + V_-}^M(x,y_1,t)$
 are semi\-equivalent. Note that  $P_+ + V_-=P_0 + V_+$, Therefore we have
\begin{multline*}\label{eq27}
C^{-1} k_{P_+}^M(x,y_1,t)\leq k_{P_+ + V_-}^M(x,y_1,t)\\ =k_{P_0+V_+}^M(x,y_1,t)\leq k_{P_0}^M(x,y_1,t) \qquad \forall x\in M, t>0.
\end{multline*}
\end{proof}





\begin{center}
{\bf Acknowledgments}
\end{center}
The author acknowledges the support of the Israel Science
Foundation (grant 963/11) founded by the Israeli Academy of Sciences and Humanities.

%
\end{document}